\documentclass[11pt,leqno] {amsart}
\usepackage{amsfonts,amssymb,amsmath,amsthm}
\usepackage{mathrsfs}
\usepackage{url}
\usepackage{enumerate}
\usepackage{xcolor, soul}

\usepackage{appendix}

\numberwithin{equation}{section}

\newtheorem{theorem}{Theorem}[section]
\newtheorem*{theorem*}{Theorem}

\newtheorem{lemma}[theorem]{Lemma}

\newtheorem{prop}[theorem]{Proposition}
\newtheorem{cor}[theorem]{Corollary}
\theoremstyle{definition}
\newtheorem{defn}[theorem]{Definition}
\newtheorem{rem}[theorem]{Remark}

\newcommand{\ncom}{\newcommand}

\ncom{\s}{\underline{s}}
\newcommand{\bq}{\begin{equation}}
\newcommand{\eq}{\end{equation}}
\newcommand{\beqn}{\begin{eqnarray*}}
\newcommand{\eeqn}{\end{eqnarray*}}
\newcommand{\inp}[2]{\langle{#1},\,{#2} \rangle}
\newenvironment{myproof}[1][\proofname]{%
  \begin{proof}[#1] \par\nobreak\ignorespaces
}{%
  \end{proof}}

\setstcolor{red}

\usepackage{todonotes}

\begin{document}
\title[Multiplication tuples homogeneous under the unitary group]{Commuting tuple of multiplication operators\\ homogeneous under the unitary group}

\author[S. Ghara]{Soumitra Ghara}
\address[S. Ghara]{Department of Mathematics and Statistics, Indian Institute of Technology Kanpur, Kanpur 208016, India}
\email[S. Ghara]{ghara90@gmail.com}
\author[S. Kumar]{Surjit Kumar}
\address[S. Kumar]{Department of Mathematics, Indian Institute of Technology Madras, Chennai 600036, India}
\email[S. Kumar]{surjit@iitm.ac.in}
\author[G. Misra]{Gadadhar Misra}
\address[G. Misra]{Statistics and Mathematics Unit, Indian Statistical Institute, Bangalore 560059, and Department of Mathematics, Indian Institute of Technology, Gandhinagar 382055}
\email[G. Misra]{gm@isibang.ac.in}
\author[P. Pramanick]{Paramita Pramanick}
\address[P. Pramanick]{Department of Mathematics and Statistics, Indian Institute of Technology Kanpur, Kanpur 208016, India}
\email[P. Pramanick]{paramitap@iitk.ac.in}
\thanks{Part of the work by S. Ghara was carried out at the Indian Institute of Technology Kanpur. Part of the work by S. Kumar, G. Misra and P. Pramanick was carried out at the Department of Mathematics, Indian Institute of Science.} 

\thanks{Support for the work of S. Ghara was provided by  the Fields-Laval post-doctoral research Fellowship, Canada and INSPIRE Faculty Fellowship (DST/INSPIRE/04/2021/002555).
Support for the work of S. Kumar was provided in the form of the Inspire Faculty Fellowship of the Science and Engineering Research Board (SERB) and NFIG grant of IIT Madras. 
Support for the work of G. Misra was provided in the 
form of the J C Bose National Fellowship, (SERB).  
Support for the research of Paramita Pramanick was provided through 
a postdoctoral Fellowship of Harish-Chandra Research Institute and NBHM postdoctoral Fellowship.}


\begin{abstract}
Let $\mathcal U(d)$ be the  group of $d\times d$ unitary matrices.  We find conditions to ensure that a $\mathcal U(d)$-homogeneous 
$d$-tuple $\boldsymbol T$ is unitarily equivalent to multiplication by the coordinate functions  on some reproducing kernel Hilbert space 
$\mathcal H_K(\mathbb B_d, \mathbb C^n) \subseteq \mbox{\rm Hol}(\mathbb B_d, \mathbb C^n)$, $n= \dim \cap_{j=1}^d \ker T^*_{j}.$
We  describe this class of $\mathcal U(d)$-homogeneous operators, equivalently, non-negative kernels $K$ quasi-invariant under the action of 
$\mathcal U(d)$. We classify  quasi-invariant kernels $K$ transforming under
$\mathcal U(d)$ with two specific choice of multipliers. A crucial ingredient of the proof is that  the group $SU(d)$ 
has exactly two inequivalent irreducible unitary representations of dimension $d$ and none in dimensions $2, \ldots , d-1$, $d\geq 3$. 
We obtain explicit criterion for  boundedness, reducibility and mutual unitary equivalence among these operators.

\end{abstract}

\maketitle

\section{Introduction}

Let $\Omega$ be an irreducible bounded symmetric domain of rank $r$ in $\mathbb C^d$ and $\rm{Aut}(\Omega)$ be the group of bi-holomorphic automorphisms on $\Omega$. 
 Let $G$ be the connected component of identity in $\rm{Aut}(\Omega)$. It is well known that $G$ acts transitively on $\Omega$. Let $\mathbb{K}$ be the subgroup of linear automorphisms in $G$. By Cartan's theorem \cite[Proposition 2, pp. 67]{RN}, $\mathbb{K}=\{\phi\in G:\phi(0)=0\}$. The group $\mathbb K$ is known to be a maximal compact subgroup of $G$ and $\Omega$ is isomorphic to $G/{\mathbb{K}}$. 
There is a natural action of $\mathbb{K}$ on $\Omega$ given by
 \[k\cdot \boldsymbol z:=\big(k_1(\boldsymbol z), \ldots, k_d(\boldsymbol z)\big),\qquad k\in \mathbb{K} \mbox{~and~} \boldsymbol z\in \Omega,\]
 where $k_1(\boldsymbol z), \ldots, k_d(\boldsymbol z)$ are linear polynomials.
 The group $\mathbb K$ also acts on a $d$-tuple $\boldsymbol{T}=(T_1,\ldots,T_d)$ of commuting bounded linear operators defined on a complex separable Hilbert space $\mathcal{H}$, naturally, via the map \[k\cdot\boldsymbol{T}:=\big(k_1(T_1, \ldots, T_d), \ldots, k_d(T_1, \ldots, T_d)\big),\,\,  k\in \mathbb{K}. \] 
 \begin{defn}[\cite{GKP}]
 	A $d$-tuple $\boldsymbol{T}=(T_1,\ldots,T_d)$ of commuting bounded linear operators on $\mathcal{H}$ is said to be $\mathbb{K}$-homogeneous if for all $k$ in $\mathbb{K}$ the operators $\boldsymbol{T}$ and $k\cdot \boldsymbol{T}$ are unitarily equivalent, that is, for all $k$ in $\mathbb{K}$ there exists a unitary operator $\Gamma(k)$  on $\mathcal{H}$ such that 
\beqn
T_j\Gamma(k)=\Gamma(k) k_j(T_1, \ldots, T_d),\qquad j=1,2,\ldots,d.
\eeqn
\end{defn}
In particular, when $\Omega$ is the Euclidean ball $\mathbb B_d$ in $\mathbb C^d,$ then $\mathbb K$ is the group of unitary linear transformations on $\mathbb C^d$ and the {\it spherical} tuples defined in \cite{CY} are nothing but $\mathcal U(d)$-homogeneous $d$-tuples.
In this paper we would be discussing $\mathcal U(d)$-homogeneous commuting $d$-tuple $\boldsymbol M$ of multiplication by coordinate functions $z_1,\ldots, z_d$ on a reproducing kernel Hilbert space $\mathcal H_K(\mathbb B_d, \mathbb C^n)$. This Hilbert space  consists of holomorphic functions defined on $\mathbb B_d$ and taking values in $\mathbb C^n$. We consider in some detail the case of $n=d$. However, without any additional effort, we set up the machinery in the much more general context of a bounded symmetric domain $\Omega$ and the maximal compact subgroup $\mathbb K$ of its bi-holomorphic automorphism group. A detailed study of $\mathbb K$-homogenous operator is underway.


Now, let 
 $D_{\boldsymbol{T}}:\mathcal{H}\to\mathcal{H}\oplus\cdots \oplus\mathcal{H}$ be the operator 
\[D_{\boldsymbol{T}}h:=(T_{1}h,\ldots,T_{d}h), \qquad h\in \mathcal{H}.\]
We note that $\ker D_{\boldsymbol T}=\cap_{i=1}^d \ker T_{i}$ is the {\it joint kernel} and 
$\sigma_p(\boldsymbol T)  = \{\boldsymbol w\in \mathbb C^d: \ker D_{\boldsymbol T -\boldsymbol w I}\not = \boldsymbol 0 \}$ is the {\it joint point spectrum} of the $d$-tuple $\boldsymbol{T}$. 
The class $\mathcal A\mathbb{K}(\Omega)$ consisting of  $\mathbb K$-homogeneous $d$-tuples of operators with the property: 
\begin{enumerate}
\item $\dim \ker D_{\boldsymbol T^*} = 1$, 
\item $\ker D_{\boldsymbol T^*}$ is cyclic for $\boldsymbol T$, and 
\item $\Omega \subseteq \sigma_p(\boldsymbol T^*)$;
\end{enumerate}
was introduced in the recent paper \cite{GKP}, see also \cite{Up}. 
Among other things, it is shown in \cite[Theorem 2.3]{GKP} that any $d$-tuple $\boldsymbol T$ in $\mathcal A\mathbb{K}(\Omega)$ must be unitarily equivalent to the $d$-tuple $\boldsymbol M$ of multiplication by the coordinate functions on a reproducing kernel Hilbert space $\mathcal H_K(\Omega) \subseteq \mbox{\rm Hol}(\Omega, \mathbb C)$ for some $\mathbb K$-invariant kernel $K$. Recall that the Hilbert space $\mathcal H_K(\Omega)$ has a direct sum decomposition $\oplus_{\s\in \vec{\mathbb Z}^r_+} \mathcal P_{\s}$, where $\vec{\mathbb Z}^r_+$ is the set of signatures: $\s:=(s_1, \ldots , s_r)\in \mathbb Z_+^r$, $s_1 \geq s_2 \geq \cdots \geq s_r\geq 0$ and $\mathcal P_{\s}$ are the irreducible components under the action of $\mathbb K$. The invariant kernel $K$ is then of the form: $K_{\boldsymbol a}(\boldsymbol z,\boldsymbol w) = \sum_{\s \in \vec{\mathbb Z}^r_+} a_{\s} E_{\s} (\boldsymbol z, \boldsymbol w)$, where $E_{\s}$ is the reproducing kernel of $\mathcal P_{\s}$ equipped with the Fischer-Fock inner product defined by $\inp{p}{q}_{\mathcal F}:=\frac{1}{\pi^d}\int_{\mathbb C^d}p(\boldsymbol z)\overline{q(\boldsymbol z)} e^{-\|\boldsymbol z\|_2^2} dm(\boldsymbol z)$. Here $dm(z)$ denotes the Lebesgue measure on $\mathbb C^d.$

The results of \cite{GKP} also show that the properties of $\boldsymbol M$ like boundedness, membership in the Cowen-Douglas class $B_1(\Omega)$, unitary and similarity orbit  etc. can be determined from the properties of the sequence $\boldsymbol a:= \{a_{\s}\}_{\s \in \vec{\mathbb Z}^r_+}$. It is therefore natural to investigate the much larger class of $d$-tuples of homogeneous operators by assuming only that $\dim \ker D_{\boldsymbol T^*}$ is finite rather than $1$, which is the main feature of the class defined below. As one might expect, we obtain a model theorem in this case also with the major difference that the kernel $K$ need not be invariant under the action of the group $\mathbb K$, instead it is \emph{quasi-invariant!} 
 

Assume that $\ker D_{\boldsymbol T^*}$ is a cyclic subspace for $\boldsymbol T$ of dimension $n$. Let $\mathcal H^{(0)}$ be the linear space $\{p(\boldsymbol T)\gamma |~ \gamma \in \ker D_{\boldsymbol T^*}, p \in \mathcal P \}$, where $\mathcal P$ is the space of complex-valued polynomials in $d$-variables.
Fix an orthonormal basis $\{\gamma_1, \ldots , \gamma_n\}$ in  $\ker D_{\boldsymbol T^*}$. For $\boldsymbol w\in \mathbb C^d$,  the point evaluation $\mbox{\rm ev}_{\boldsymbol w}: \mathcal H^{(0)} \to \mathbb C^n$ is defined to be the map $$\mbox{\rm ev}_{\boldsymbol w} \big (\sum_{i=1}^n p_i(\boldsymbol T)(\gamma_i) \big ) := \sum_{i=1}^n p_i(\boldsymbol w) \boldsymbol e_i,$$
where $p_1,\ldots , p_n$ are in $\mathcal P$ and $\boldsymbol e_1, \ldots , \boldsymbol e_n$ are the standard unit vectors in $\mathbb C^n$. Let $\mbox{\rm bpe}(\boldsymbol T)$ be the set $\{\boldsymbol w\in \mathbb C^d: \mbox{\rm ev}_{\boldsymbol w}~ \mbox{\rm is bounded}\}$ (see \cite[Definition 2.1]{TrS}).

\begin{defn} \label{defn:1.4} Let $\Omega$ be an irreducible bounded symmetric domain. A $\mathbb{K}$-homogeneous $d$-tuple  $\boldsymbol T$ possessing the following properties  
\begin{enumerate}
\item[(i)] $\dim \ker D_{\boldsymbol T^*} = n$, 
\item[(ii)] the space $\ker D_{\boldsymbol T^*}$ is  cyclic for $\boldsymbol T$,
\item[(iii)] $\Omega \subseteq \mbox{\rm bpe}(\boldsymbol T)$, and the evaluation maps $\mbox{\rm ev}_{\boldsymbol w}$ are locally uniformly bounded for $\boldsymbol w\in \Omega$,  
\end{enumerate}
is said to be in the class $\mathcal A_n\mathbb{K}(\Omega)$. 
\end{defn}
The local uniform boundedness of the evaluation functionals might appear to be a strong requirement but is necessary for constructing a model for $d$-tuples in $\mathcal A_n\mathbb K(\Omega)$ with $n > 1$ (see proof of Theorem 2.1). This notion appears in the definition of quasi-free modules introduced in \cite{quasi}. The notion of sharp kernels (see \cite{Om}) and generalized Bergman kernels (see \cite{Curtosalinas}) occurring in the work of Agrawal-Salinas and Curto-Salinas are closely related to the kernels implicit in Definition \ref{defn:1.4}.

It follows from \cite[Theorem 2.3]{GKP} that the $d$-tuples in the class $\mathcal A\mathbb K (\Omega)$ introduced earlier in \cite{GKP} coincides with to the class $\mathcal A_1\mathbb K (\Omega)$. 
It would be convenient for us to let $\mathcal A\mathbb K (\Omega)$ denote the class $\mathcal A_1\mathbb K (\Omega)$.
In this paper, we continue the investigation initiated in \cite{GKP}, now for the class $\mathcal A_n\mathbb K (\Omega)$, $n >1$. 



\begin{defn} \label{qinv}
Let $K:\Omega \times \Omega \to \mathcal M_n(\mathbb C)$ be a sesqui-analytic 
Hermitian function and $c:\mathbb K \times \Omega \to \mbox{GL}_n(\mathbb C))$
be a function holomorphic in the second variable for each fixed $k \in \mathbb K$. 
The function $K$ is said to be quasi-invariant under the group $\mathbb K$ with multiplier $c$ if  
 $$K( \boldsymbol z , \boldsymbol w) = c(k,\boldsymbol z) K( k^{-1} \cdot\boldsymbol z, k^{-1} \cdot\boldsymbol w) {c(k,\boldsymbol w)^*}, \,\, k\in \mathbb K.$$
\end{defn}

We point out that if the function $K$ is quasi-invariant and non-negative definite, then the map $\Gamma(k)$, $k \in \mathbb K$ defined by the rule: $\Gamma(k)(f)= c(k, \boldsymbol z) f\circ k^{-1}$ is unitary on the reproducing kernel Hilbert space $\mathcal H_K(\Omega, \mathbb C^n)$.  Also, the map $k \to \Gamma(k)$ is a homomorphism if and only if $c$ is a cocycle, that is, 
$$c(k_1 k_2, \boldsymbol z) = c(k_1, k_2 \cdot \boldsymbol z) c(k_2, \boldsymbol z), \, k_1, k_2 \in \mathbb K.$$

In the explicit examples we discuss, the map 
$c:\mathbb K \times \Omega \to \mbox{GL}_n(\mathbb C)$ is constant in the second variable and therefore defines a unitary representation of the group $\mathbb K$. These examples consist of $\Omega=\mathbb B_d$ and $c(k) = k$ or $c(k) = \bar{k}$, $k\in \mathbb K$, which in this case is $\mathcal U(d)$.
Consequently,  the intertwining operator $\Gamma(k)$ defines a  unitary representation $k\to \Gamma(k)$ of the group $\mathbb K$.  Indeed, if there is a unitary $\Gamma(k)$, $k\in \mathbb K$, intertwining $\boldsymbol M$ and $k \cdot \boldsymbol M$, then the reproducing kernel $K$ must be quasi-invariant. A familiar argument using the very useful notion of ``normalized kernel'', see Remark \ref{normalized}, then shows that the function $c$ must be actually independent of $\boldsymbol z$. What is more, it is also shown that $c(k)$ is unitary for each $k\in \mathbb K$. 


If the $d$-tuple $\boldsymbol M$ on some Hilbert space $\mathcal H_K(\Omega)$ is in  $\mathcal A\mathbb K (\Omega)$, then the kernel $K$ is invariant under the action of the group $\mathbb K$, that is, $K(\boldsymbol z,\boldsymbol w) =\sum_{\underline s\in \vec{\mathbb Z}^r_+} a_{\underline s} E_{\underline s}(\boldsymbol z, \boldsymbol w)$ with $a_0=1$, see  \cite[Proposition 3.4]{Arazy} and \cite[Theorem 2.3]{GKP}.  But if $n > 1$ and the $d$-tuple $\boldsymbol M$ acting on $\mathcal H_K(\Omega, \mathbb C^n)$ is in $\mathcal A_n\mathbb K(\Omega)$, then we can only assume that the kernel $K$ is merely quasi-invariant, not necessarily invariant. 
How do we construct, if there is any, an example of a kernel $K: \Omega \times \Omega \to \mathcal M_n(\mathbb C)$ which is quasi-invariant but not invariant. Equivalently, we are asking: If $\boldsymbol M$ is in $\mathcal A_n\mathbb{K}(\Omega)$ acting on the Hilbert space $\mathcal H_K(\Omega, \mathbb C^n)$ ($n >1$), then does it follow that the quasi-invariant  kernel $K$ must be necessarily invariant? Consider, for example, the kernel 
$$ \mathcal K_{\boldsymbol a}(\boldsymbol w, \boldsymbol w):= K_{\boldsymbol a}^{2}(\boldsymbol w, \boldsymbol w) \Big(\!\!\Big( \frac{\partial^2}{\partial w_i {\partial} \bar{w}_j} \log K_{\boldsymbol a} (\boldsymbol w,\boldsymbol w)\Big)\!\!\Big),$$
where $K_{\boldsymbol a}: \Omega \times \Omega \to \mathbb C$ is an invariant positive definite kernel of the form $K_{\boldsymbol a}(\boldsymbol z,\boldsymbol w) =\sum_{\underline s\in \vec{\mathbb Z}^r_+} \! a_{\s} E_{\s}(\boldsymbol z, \boldsymbol w)$. It is known that $\mathcal K_{\boldsymbol a}$ is not only a positive definite kernel but also quasi-invariant under $\mathbb K$, see \cite[Proposition 2.3 and Proposition 6.2]{GM}. 
Indeed, $\mathcal K_{\boldsymbol a}$ transforms according to the rule:
$${k^{-1}}^{\dagger} \mathcal K_{\boldsymbol a}(k^{-1} \cdot \boldsymbol z, k^{-1} \cdot \boldsymbol w) \overline{k^{-1}} =  \mathcal K_{\boldsymbol a} (\boldsymbol z, \boldsymbol w) , \,\, k\in \mathbb  K,$$
where $\dagger$ denotes the transpose of a matrix. 
The multiplier $c:\mathbb K \times \Omega \to {\rm GL}_d(\mathbb C)$ for the quasi-invariant kernel $\mathcal K_{\boldsymbol a}$ is given by $c(k,\boldsymbol z) = \overline{k},\, k\in \mathbb K,\, z\in \Omega$. It is not hard to see that $\mathcal K_{\boldsymbol a}$ is \emph{not} invariant under $\mathbb K$, see Proposition \ref{qi}.
Thus, we have many examples of quasi-invariant kernels taking values in $\mathcal M_n(\mathbb C)$ that are not invariant when $n=d$. We briefly describe below the results of this paper.
 

In Section 2, we find a concrete model for a $d$-tuple $\boldsymbol T$ in $\mathcal A_n \mathbb{K}(\Omega)$ as the $d$-tuple $\boldsymbol M$ of multiplication by the coordinate functions $z_1,\ldots,z_d$ on some Hilbert space $\mathcal H_K(\Omega, \mathbb C^n) \subseteq \mbox{\rm Hol} (\Omega, \mathbb C^n)$ possessing a reproducing kernel $K:\Omega \times \Omega \to \mathcal M_n(\mathbb C)$. This is Theorem \ref{model}.
We prove, see Theorem \ref{3main}, that a quasi-invariant kernel $K$ is a sum (with positive coefficients) of certain quasi-invariant kernels in the Peter-Weyl decomposition of the Hilbert space $\mathcal H_K(\Omega, \mathbb C^n)$ with respect to the action of the group $\mathbb K$.  
%

In Section 3, we restrict to the case of the Euclidean ball $\mathbb B_d \subseteq \mathbb C^d$. Designating $\pi_\ell$ the natural action of $\mathcal U(d)$ on the homogeneous polynomials of degree $\ell$ in $d$ variables equipped with the Fisher-Fock inner product. We prove that $\pi_1\otimes \pi_\ell$ is reducible and identify an irreducible  component in the decomposition of $\pi_1\otimes \pi_\ell$. We obtain a similar result for $\bar{\pi}_1\otimes \pi_\ell$, where $\bar{\pi}_1$ is the  contragredient of $\pi_1$. 
 Choosing the cocycles $c(u,z) = \pi_1(u)$, its contragredient $c(u,z) = \bar{\pi}_1(u)$, $u\in \mathcal U(d)$, we describe all the sesique-analaytic Hermitian quasi-invariant function that transform as in Definition \ref{qinv}. 
Among these, the non-negative definite functions are identified explicitly. We conclude by discussing  two sets of examples of $d$-tuples in $\mathcal A_d\mathcal U (\mathbb B_d)$. 

In the first half of Section 4, we find conditions for boundedness and irreducibility of the $d$-tuple $\boldsymbol M$. The second half is devoted to study of quasi-invariant diagonal kernels $K:\mathbb B_d \times \mathbb B_d \to \mathcal M_n(\mathbb C)$. In this case, such a kernel must be invariant and we prove that 
it is of the form:
$\sum_{\ell=0}^\infty A_\ell \inp{\boldsymbol z}{\boldsymbol w}^\ell, \boldsymbol z, \boldsymbol w\in \mathbb B_d$, see Corollary \ref{diagonalkernel}. 

In the concluding Section 5, first, we identify the two components in the decomposition of $\pi_1\otimes \pi_\ell$ (respectively, $\bar{\pi}_1\otimes \pi_\ell$) explicitly and show that these components themselves are irreducible. Secondly, we prove that if a kernel $K$ is quasi-invariant under $\mathcal U(d)$ taking values in $\mathcal M_d(\mathbb C)$, transforms as in Definition \ref{qinv} with $c:\mathcal U(d) \to {\rm GL}_d(\mathbb C)$, and $c$ is assumed to be an irreducible representation of $\mathcal U(d)$, then these kernels fall into two classes explicitly described in Theorem \ref{class}. To prove this result, we first establish that, up to unitary equivalence, there are only two irreducible unitary representations of $SU(d)$, the standard one and its contragredient. We also prove that $SU(d)$ does not have any irreducible unitary representation of dimension $\ell$, $2\leq \ell \leq d-1$. We were not able to locate these results that might be of independent interest. Therefore, we have included detailed proofs of these results. 

For now, we have complete results only in the particular case of the cocycles $c(u,z) = u$ or $\bar{u}$ of the group $\mathcal U(d)$, $d\in \mathbb N$. We are hopeful of obtaining similar results for an arbitrary cocycle in the case of the group $\mathcal U(2)$. 




\section{Decomposition of a quasi-invariant kernel}
We begin by providing a model for a $d$-tuple of operator $\boldsymbol T$ in the class  $\mathcal A_n\mathbb K(\Omega)$ acting on some Hilbert space $\mathcal H$. The proof involves transplanting the inner product of $\mathcal H$ on the subspace $\mathbb C^n \otimes \mathcal P$ of $\mathbb C^n$-valued polynomials in the space of holomorphic functions $\text{Hol}(\Omega, \mathbb C^n)$. The proof amounts to constructing a unitary operator intertwining $\boldsymbol T$ and the $d$-tuple of multiplication operators defined on the completion of the subspace $\mathbb C^n \otimes \mathcal P$ in $\text{Hol}(\Omega, \mathbb C^n)$. 
\begin{theorem}\label{model}
 Suppose that $\boldsymbol T$ is a $d$-tuple of commuting operators in $\mathcal A_n\mathbb K(\Omega)$. Then $\boldsymbol T$ is unitarily equivalent to the  $d$-tuple $\boldsymbol M$ of multiplication by the coordinate functions $z_1,\ldots,z_d$ on a reproducing kernel Hilbert space $\mathcal H_K(\Omega, \mathbb C^n) \subset \text{\rm Hol}(\Omega, \mathbb C^n)$, for some kernel function $K$  quasi-invariant under $\mathbb K$.
\end{theorem}
\begin{proof}
Since $\boldsymbol{T}$ is $\mathbb{K}$-homogeneous, for each $k\in \mathbb{K}$ there exists a unitary operator $\Gamma(k)$ on $\mathcal H$ such that
\[ T_j\Gamma(k)=\Gamma(k) k_j(\boldsymbol{T}), \qquad j=1,\ldots,d.\]

Pick an orthonormal basis $\{\boldsymbol \xi_1, \ldots , \boldsymbol \xi_n\} \subseteq \ker D_{\boldsymbol{T}^*}$. 
Let $\iota :\ker D_{{\boldsymbol T}^*} \to \mathbb C^n$ be a unitary identifying $\boldsymbol \xi = \sum_{i=1}^n  x_ i\boldsymbol \xi_i$ with the vector $\boldsymbol x = \sum_{i=1}^n x_i \boldsymbol e_i$, where $\boldsymbol e_1, \ldots , \boldsymbol e_n$ are the standard unit vectors in $\mathbb C^n.$ 
We define a semi-inner product on $\mathbb C^n \otimes \mathcal P$ by extending 
\begin{equation}\label{Surjit}
\inp{\boldsymbol e_i \otimes p }{\boldsymbol e_j \otimes q}:= \inp{p(\boldsymbol T)\boldsymbol \xi_i }{q(\boldsymbol T)\boldsymbol \xi_j}_{\mathcal H}, p, q \in \mathcal{P},
\end{equation}
to $\mathbb C^n \otimes \mathcal P$ by linearity. Suppose that $\big \|\sum_{i=1}^n  \boldsymbol e_i \otimes p_i  \big \| = 0$, then we claim that $\sum_{i=1}^n \boldsymbol e_i \otimes  p_i  =0$. Pick any $w\in \Omega \subseteq \mbox{\rm bpe}(\boldsymbol T)$ and note that 
$$ \big \|\sum_{i=1}^n p_i(w) \boldsymbol e_i \big \|_2  \leq C_w \big \|\sum_{i=1}^n p_i (\boldsymbol T) \boldsymbol \xi_i \big \|_\mathcal H = 0.$$
For $1\leq i \leq n$, it follows that $p_i(\boldsymbol w) =0$ for all $\boldsymbol w \in \Omega$. Consequently each $p_i$, $1\leq i \leq n$, is the zero polynomial. Therefore, the semi-inner product given by the formula \eqref{Surjit} defines an inner product on $\mathbb C^n \otimes \mathcal P$. 

Let $\mathscr H$ be the completion of $\mathbb C^n \otimes \mathcal P$ with respect to this inner product. Since we have assumed that the set $\mbox{\rm bpe}(\boldsymbol T) $ contains $\Omega$, it follows that the Hilbert space $\mathscr H$ is a reproducing kernel Hilbert space consisting of functions defined on $\Omega$. Let $K:\Omega \times \Omega \to \mathcal M_n(\mathbb C)$ be the kernel function given by $K(\boldsymbol z,\boldsymbol w)=\mbox{ev}_{\boldsymbol z} \, \mbox{ev}_{\boldsymbol w}^*$, that is, 
\begin{enumerate}
    \item $K (\cdot, \boldsymbol w) \boldsymbol x$ is in $\mathscr H$ for every vector $\boldsymbol x \in \mathbb C^n$ and every point $\boldsymbol w\in \Omega,$
    \item $\inp{f} {K (\cdot, \boldsymbol w) \boldsymbol x }_{\mathscr H} =\inp{f(\boldsymbol w)}{ \boldsymbol x}_{2}$.
\end{enumerate}
Given any function $f\in \mathscr H$, we can find polynomials $p_j\in \mathbb C^n\otimes \mathcal P$ such that $\|f-p_j\|_\mathscr H \to 0$ as $j \to \infty$ by assumption. 
Moreover, since the point evaluations are assumed to be locally uniformly bounded on $\Omega$, it follows that for any fixed but arbitrary $\boldsymbol w \in \Omega $, there is an open set $\mathcal O \subseteq \Omega$ containing $\boldsymbol w$ such that $\sup_{\boldsymbol z\in \mathcal O} \|K(\boldsymbol z, \boldsymbol z)\| = N_{\mathcal O, \boldsymbol w}< \infty$. 
For any compact set $X\subseteq \mathcal O$, and $\boldsymbol z\in X$, we have 
\begin{equation}
|\langle f(\boldsymbol z)-p_j(\boldsymbol z), \boldsymbol e_i\rangle|\leq \|f(\boldsymbol z) - p_j(\boldsymbol z)\|_2 \leq  N_{\mathcal O,\boldsymbol w}^{1/2} \|f - p_j\|_{\mathscr H}
    \end{equation}
proving that $f$ is holomorphic at $\boldsymbol w$. 
Consequently, $K$ is holomorphic in the first variable and anti-holomorphic in the second. 

Now for any $k\in \mathbb{K}$, since $\ker D_{\boldsymbol T^*}$ is invariant under the unitary map $\Gamma(k)^*$, we have 
\begin{align*} \inp{ \boldsymbol e_i \otimes p}{\boldsymbol e_j \otimes q}_{\mathbb C^n \otimes \mathcal{P} } 
&=\inp{p (\boldsymbol T)\boldsymbol \xi_i}{q(\boldsymbol T)\boldsymbol \xi_j}_{\mathcal H}\\
&=\inp{\Gamma(k) p (k \cdot \boldsymbol T)\Gamma(k)^*\boldsymbol \xi_i}{\Gamma(k)q(k \cdot \boldsymbol T)\Gamma(k)^*\boldsymbol \xi_j}_{\mathcal H}\\
&= \inp{p (k \cdot \boldsymbol T)\Gamma(k)^*\boldsymbol \xi_i}{q(k \cdot \boldsymbol T)\Gamma(k)^*\boldsymbol \xi_j}_{\mathcal H}\\ 
&= \inp{\iota \Gamma(k)^* \iota^* \boldsymbol e_i \otimes p\circ k }{\iota \Gamma(k)^* \iota^* \boldsymbol e_j \otimes q \circ k }_{\mathbb C^n \otimes \mathcal{P}}. \end{align*}
Therefore, the reproducing kernel $K$ of the Hilbert space $\mathscr H$ is quasi-invariant under $\mathbb K$ with multiplier $\iota \Gamma(k)^* \iota^*$. 
Finally, the unitary taking $\boldsymbol e_i \otimes p $ to $p(\boldsymbol T) \boldsymbol \xi_i$ extends to a unitary from the Hilbert space $\mathcal H$ to the Hilbert space $\mathscr H.$ This unitary intertwines the commuting $d$-tuple $\boldsymbol T$ on $\mathcal H$ with the $d$-tuple $\boldsymbol M$ of multiplication by the coordinate functions $z_i$, $1\leq i \leq d$, on $\mathscr H$.  
\end{proof}

Now we gather a few properties of $d$-tuples in the class $\mathcal A_n \mathbb K(\Omega)$. In particular, we prove that if the $d$-tuple $\boldsymbol M$ on $\mathcal H_K(\Omega, \mathbb C^n)$ is in $\mathcal A_n\mathbb K(\Omega)$, then the  
intertwining unitary between $\boldsymbol M$ and $k \cdot \boldsymbol M$ for each $k\in \mathbb K$ must be of the form $f \to c(k) (f\circ k^{-1})$, $c(k) \in \mathcal U(n)$. 
\begin{rem} \label{normalized} 
We recall that any non-negative definite kernel $K:\Omega \times \Omega \to \mathcal M_n(\mathbb C)$ admits a normalization $K_0$ at $ \boldsymbol w_0 \in \Omega$. The normalized kernel $K_0$ is characterized by the requirement $K_0(\boldsymbol z,\boldsymbol w_0) =  \text{Id}_n$ for all $\boldsymbol z \in \Omega$. The point $\boldsymbol w_0$ is arbitrary but fixed. 
The first two of the three statements below can be found in \cite{Curtosalinas} and the last one is from \cite[p. 285, Remark]{equivofquotient}. 
\begin{enumerate}
\item The $d$-tuple $\boldsymbol M$ on $\mathcal H_K(\Omega, \mathbb C^n)$ and $\mathcal H_{K_0}(\Omega, \mathbb C^n)$ are unitarily equivalent. 
\item If $K_1$ and $K_2$ be the kernels normalized at some fixed $\boldsymbol w_0 \in \Omega,$ then the multiplication $d$-tuples  on $\mathcal H_{K_1}(\Omega, \mathbb C^n)$ and $\mathcal H_{K_2}(\Omega, \mathbb C^n)$ are unitarily equivalent if and only if there is a unitary $U\in \mathcal U(n)$ such that $U^*K_1(\boldsymbol z,\boldsymbol w)U = K_2(\boldsymbol z,\boldsymbol w)$ for all $\boldsymbol z,\boldsymbol w \in \Omega$.
\item Suppose that $\mathbb C^n \otimes \mathcal P$ is densely contained in  $\mathcal H_K(\Omega, \mathbb C^n)$ and that the multiplication by the coordinate functions are bounded on $\mathcal H_K(\Omega, \mathbb C^n)$. Then 
$$\cap_{i=1}^n \ker(M_i - w_i)^* = \{K(\cdot , \boldsymbol w) \boldsymbol{x}: \boldsymbol{x} \in \mathbb C^n\}.$$
Moreover, the dimension of the joint kernel at $\boldsymbol w$ is $n$ for all $\boldsymbol w\in \Omega$. 
\end{enumerate} 
\end{rem}

\begin{lemma} \label{multiplier}
Let $\mathcal H_K(\Omega, \mathbb C^n)$ be a reproducing kernel Hilbert space consisting of holomorphic functions on $\Omega$ taking values in $\mathbb C^n$. Assume that  $\mathbb C^n \otimes \mathcal P$ is densely contained in $ \mathcal H_K(\Omega, \mathbb C^n)$, the $d$-tuple $\boldsymbol M$ on $\mathcal H_K(\Omega, \mathbb C^n)$ is bounded and the kernel $K$ is normalized at $0$. Then the following statements are equivalent. 
\begin{enumerate}
\item The $d$-tuple $\boldsymbol M$ is $\mathbb K$-homogeneous, that is, there is a unitary operator $\Gamma(k)$ on $\mathcal H_K(\Omega, \mathbb C^n)$ with  
$$\Gamma(k) (k \cdot \boldsymbol M ) \Gamma(k)^* = \boldsymbol M, \,\,k\in\mathbb K.$$ 
\item The kernel $K$ is quasi-invariant under $\mathbb K$ with multiplier $c:\mathbb K \times \Omega \to \mathcal U(n)$, $c(k,\boldsymbol z)$ is independent of $\boldsymbol z$. 
\item 
There is a map $c:\mathbb K \to \mathcal U(n)$ such that $(\Gamma(k) f)(\boldsymbol z) := c(k) f( k^{-1} \cdot \boldsymbol z)$, is unitary on $\mathcal H_K(\Omega, \mathbb C^n)$. 
\end{enumerate}
\end{lemma}
\begin{proof} Since $\mathbb C^n \otimes \mathcal P$ is densely contained in $\mathcal H_K(\Omega, \mathbb C^n)$, it follows that the dimension of the  joint kernels $\cap_{i=1}^d \ker D_{(\boldsymbol M - w)^*}$, $w\in \Omega$, as shown in \cite[p. 285, Remark]{equivofquotient}, is $n$. Therefore, the methods of \cite{Curtosalinas} applies. 

First,  it is not hard to see that the $d$-tuple of operators $k\cdot \boldsymbol M$ acting on the Hilbert space $\mathcal H_K(\Omega, \mathbb C^n)$ is unitarily equivalent to the $d$-tuple $\boldsymbol M$ acting on $\mathcal H_{\hat{K}}(\Omega, \mathbb C^n)$, where $\hat{K}(\boldsymbol z,\boldsymbol w) := K(k^{-1}\cdot \boldsymbol z, k^{-1} \cdot \boldsymbol w)$ via the unitary operator $f \to f\circ k^{-1}$, $f\in \mathcal H_K(\Omega, \mathbb C^n)$. 
Since $K$ is assumed to be normalized at $0$ and $k$ is linear, it follows that $\hat{K}$ is also normalized at $0$. 
Second, for a fixed $k\in \mathbb K$, any intertwining unitary operator between the $d$-tuple $\boldsymbol M$ on $\mathcal H_{\hat{K}}(\Omega, \mathbb C^n)$ and $\mathcal H_{K}(\Omega, \mathbb C^n)$ must be of the form $\hat{f} \to c(k)\hat{f}$, where $(c(k) \hat{f})(\boldsymbol z) = c(k) \hat{f}(\boldsymbol z)$ for some unitary $c(k) \in \mathcal U(n)$. 
Finally, these two unitaries combine to give a unitary operator $\Gamma(k): \mathcal H_K(\Omega, \mathbb C^n) \to \mathcal H_K(\Omega, \mathbb C^n)$ of the form: $\Gamma(k)f (\boldsymbol z) = c(k) (f\circ k^{-1})(\boldsymbol z)$. Thus 
we have proved that the statement (1) implies (3). 

Moreover, the unitarity of the map $\Gamma$ in the statement (3) is equivalent to the quasi-invariance of the kernel K, namely, $K(\boldsymbol z,\boldsymbol w) = c(k) K(k^{-1} \cdot \boldsymbol z, k^{-1} \cdot \boldsymbol w) c(k)^*$. This proves the equivalence of the statements (2) and (3). 

The statement (3) clearly implies (1) completing the proof. 
\end{proof}
\begin{rem}
Choosing the multiplier $c:\mathbb K \to {\rm GL}_n(\mathbb C)$ to be unitary without loss of generality and assuming that $c$ is a homomorphism, we see that the map $f \to c(k) (f \circ k^{-1})$ is a unitary representation of $\mathbb K$ on the Hilbert space $\mathcal H_K(\Omega, \mathbb C^n)$. 
\end{rem}

The group $\mathbb K$ acts on $\mathcal P$ naturally by the rule $p \to p\circ k^{-1}$. This action, as is well known, decomposes into irreducible components $\mathcal P_{\underline{s}}$ parametrized by the signatures $\underline{s}$ in $\vec{\mathbb Z}^r_+$. It is pointed out in \cite[Proposition 3.4]{Arazy}, that any $\mathbb K$-invariant inner product on $\mathcal P$ must be of the form 
$$\langle p, q \rangle = \sum_{\ell=0}^{\text{deg}~ p}\sum_{\substack{|\s|=\ell \\ \s \in \vec{\mathbb Z}^r_+}}a_{\underline{s}}\langle p_{\underline{s}}, q_{\underline{s}}\rangle_{\mathcal F},$$
where deg $p$ is the degree of $p$ and $p_{\underline{s}}$, $q_{\underline{s}}$ are the components of $p,q\in \mathcal P$ in the Peter-Weyl decomposition of $\mathcal P$ into irreducible subspaces $\mathcal P_{\underline{s}}$. In this paper, what we study amounts to finding $\mathbb K$ quasi-invariant inner products on the space $\mathbb C^n\otimes \mathcal P$. We do this by obtaining a generalization of the description of an invariant inner product from the scalar case given above. This is Proposition \ref{Peter-Weyl}.  For the proof, we need the following elementary lemma (compare with Lemma 2.8 of \cite{CY}). 

\begin{lemma}\label{Schurs lemma}
Let $\mathcal H_1:=(\mathcal H, \langle \cdot , \cdot \rangle_1)$ and $\mathcal H_2:=(\mathcal H, \langle \cdot , \cdot \rangle_2)$ be two Hilbert spaces. Let $\rho:\mathbb K\to 
\mathcal U({\mathcal H}_i)$ be an irreducible unitary representation for $i=1,2$. Then there exists a positive scalar $\delta$ such that $\langle \cdot, \cdot\rangle_1= \delta \langle \cdot, \cdot\rangle_2$. 
\end{lemma}

\begin{proof}
Let $A$ be the linear map from $\mathcal H$ to $\mathcal H$ such that $\langle f, g\rangle_{\mathcal H_1}=\langle Af, g\rangle_{\mathcal H_2}.$ Now, note that,
\begin{align*}
    \langle \rho(k)Af, g\rangle_{\mathcal H_2}&=\langle Af,\rho(k^{-1})g\rangle_{\mathcal H_2}\\&=\langle f,\rho(k^{-1})g\rangle_{\mathcal H_1}\\&=\langle \rho(k)f,g\rangle_{\mathcal H_1}\\
    &=\langle A\rho(k)f,g\rangle_{\mathcal H_2}
\end{align*}
Thus it follows that $\rho(k)A=A\rho(k)$. An application of Schur's lemma completes the proof.
\end{proof}

Let $\pi$ be a unitary representation of the compact group $\mathbb K$ on a Hilbert space $\mathcal H$ containing $\mathbb C^n \otimes \mathcal P$ as a dense subspace. By the Peter-Weyl theorem, $\mathcal H$  is the direct sum of  irreducible representations of $\mathbb K$ acting on finite dimensional subspaces $\mathcal H_\lambda$, $\lambda\in \Lambda$. 
Let $\pi_\lambda$ be the restriction of $\pi$ to $\mathcal H_\lambda$, that is, $\pi = \oplus_{\lambda \in \Lambda} \pi_\lambda$ is the Peter-Weyl decomposition relative to the direct sum $\mathcal H = \oplus_{\lambda\in \Lambda} \mathcal H_\lambda$ into reducing subspaces of $\pi$. For the complete statement of the Peter-Weyl theorem one may consult \cite[Theorem, 1.12, p.  17]{knapp}.

Let us transplant the Fischer-Fock inner product on $\mathcal P$ and the Euclidean inner product on $\mathbb C^n$ to the tensor product   $\mathbb C^n \otimes \mathcal P$. We let $\inp{\cdot}{\cdot}_\mathcal F$ denote the inner product on this tensor product space by a slight abuse of notation. Let  $P_\lambda$ be the linear subspace of  $\mathbb C^n \otimes \mathcal P$ identified with $\mathcal H_\lambda$. Now, each of the subspaces $P_\lambda \subset \mathbb C^n \otimes \mathcal P$ inherits the inner product from that of $(\mathbb C^n \otimes \mathcal P , \inp{\cdot}{\cdot}_\mathcal F)$ to be denoted by $(P_\lambda, \inp{\cdot}{\cdot}_{\mathcal F_\lambda})$, $\lambda \in \Lambda$.  
The hypothesis in the following proposition might appear to be restrictive but for the applications in this paper, they appear naturally. 

\begin{prop}\label{Peter-Weyl}
Fix an action $\pi$ of the compact group $\mathbb K$ on a Hilbert space $\mathcal H$. Let $[\cdot, \cdot]$ denote the inner product of $\mathcal H$.  Assume that $ \mathbb C^n \otimes \mathcal P$ is a dense subspace of $\mathcal H$. 
Furthermore, we assume that (a) 
$[p, q] = [\pi(k) p,  \pi(k) q]$,  that is, $\pi$ is a unitary representation of $\mathbb K$ on $\mathcal H$ (b) $\langle p_\lambda,  q_\lambda \rangle_{\mathcal F_\lambda} = \langle \pi_\lambda(k) p_\lambda,  \pi_\lambda(k) q_\lambda \rangle_{\mathcal F_\lambda}$, $k\in \mathbb K$, (c) $\pi_\lambda$ and $\pi_{\lambda^\prime}$ are inequivalent whenever $\lambda \not = \lambda^\prime$. 
Then there exists positive scalars $a_\lambda$ such that $[p, q] = \sum_{\lambda \in \Lambda} a_\lambda \inp{p_\lambda}{q_\lambda}_{ \mathcal F_\lambda}$, where $ p= \sum_{\lambda \in \Lambda} p_\lambda$ and $q= \sum_{\lambda \in \Lambda} q_\lambda$,  $p, q \in \mathbb C^n \otimes \mathcal P$. 
\end{prop}
\begin{proof} Let $p, q \in \mathbb C^n \otimes \mathcal P$ be 
of the form $\sum_{\lambda\in \Lambda} p_\lambda$, $p_\lambda \in P_\lambda$, and 
$\sum_{\lambda\in \Lambda} q_\lambda$, $q_\lambda \in P_\lambda$, respectively.
For any pair $\lambda \not = \lambda^\prime$, by hypothesis, $\pi_\lambda$ and $\pi_{\lambda^\prime}$ are inequivalent, therefore the subspaces $P_\lambda$ and $P_{\lambda^\prime}$ of the inner product space  $(\mathbb C^n \otimes \mathcal P, [\cdot, \cdot ])$ are orthogonal. Therefore, we have 
$$[ p, q] = \sum_{\lambda\in \Lambda} [p_\lambda, q_\lambda].$$ 
The representation $\pi_\lambda$ of $\mathbb K$ on  $(P_\lambda, [\cdot,\cdot])$ is unitary and irreducible. It is also unitary and irreducible on the space  $(P_\lambda, \inp{\cdot}{\cdot}_{\mathcal F_\lambda})$. The proof of the theorem is completed by applying Lemma \ref{Schurs lemma}. 
\end{proof}
As an application of Proposition \ref{Peter-Weyl}, we obtain a description of all the quasi-invariant kernels $K$ with a multiplier $c$ assuming that $c$ is a unitary representation   of $\mathbb K$.  
\begin{theorem}\label{3main}
Let $\mathcal H_K(\Omega, \mathbb C^n)$ be a reproducing kernel Hilbert space densely containing $\mathbb C^n \otimes \mathcal P$ as subspace. 
Assume that $K$ is  quasi-invariant with multiplier $c$, where $c:\mathbb K \to \mathcal U(n)$ is a representation of the group $\mathbb K$.  Let $\pi$ denote the action of the group $\mathbb K$ on $\mathcal H_K(\Omega, \mathbb C^n)$ given by the rule $\pi(k) f = c(k) (f\circ k^{-1})$.
In the Peter-Weyl decomposition $\pi = \oplus_{\lambda\in \Lambda} \pi_\lambda$, assume that the unitary representations $\pi_\lambda$ are inequivalent. Then there exists positive scalars $b_{\lambda}, \lambda\in \Lambda$, such that $$K(\boldsymbol z, \boldsymbol w)=\sum_{\lambda\in \Lambda}b_{\lambda}K_{\lambda}(\boldsymbol z, \boldsymbol w),~\boldsymbol z, \boldsymbol w\in \Omega,$$ where $K_{\lambda}$ is the reproducing kernel of $(P_{\lambda}, \inp{\cdot}{\cdot}_{\mathcal F_{\lambda}}),$ and  $\mathcal H_K(\Omega, \mathbb C^n) =\oplus_{\lambda \in \Lambda} P_{\lambda}.$
\end{theorem}
\begin{proof}
From Lemma \ref{multiplier}, it follows that the action $\pi$
of the group $\mathbb K$ on $\mathcal H_K(\Omega, \mathbb C^n)$ is unitary. This verifies the assumption (a) of Proposition \ref{Peter-Weyl}. The inner product space  $(\mathbb C^n \otimes \mathcal P, \inp{\cdot}{\cdot}_\mathcal F)$ is the tensor product  $(\mathbb C^n, \inp{\cdot}{\cdot}_2 ) \otimes (\mathcal P_\lambda, \inp{\cdot}{\cdot}_\lambda)$. Consequently, since $c(k)$ is unitary for each $k\in \mathbb K$ verifying assumption (b) of Proposition \ref{Peter-Weyl}. Finally, 
the assumption that $\pi_\lambda$, $\lambda \in \Lambda$, are inequivalent is the assumption (c) of Proposition \ref{Peter-Weyl}. Therefore the proof is completed by applying Proposition \ref{Peter-Weyl}.  
\end{proof}
We show that a non-scalar kernel $K$, quasi-invariant under $\mathcal U(d)$ associated with a multiplier $c$ that is irreducible, can not be invariant. 
\begin{prop}\label{qi}
Let $K:\Omega \times \Omega \to \mathcal M_n(\mathbb C)$ be a non-negative definite kernel. 
Suppose that $c: \mathbb K \to \mathcal M_n(\mathbb C)$ is an irreducible unitary representation and $K$ is quasi-invariant under $\mathbb K$ with multiplier $c$.  If the kernel $K$ is also invariant under $\mathbb K$, then there exists a non-negative definite scalar valued kernel $\kappa$ on $\Omega\times\Omega$ invariant under $\mathbb K$ such that $K(\boldsymbol z, \boldsymbol w)=\kappa(\boldsymbol z, \boldsymbol w)I_n$, $\boldsymbol z, \boldsymbol w\in \Omega.$
\end{prop}
\begin{proof}
Suppose that $K$ is quasi-invariant with multiplier $c: \mathbb K \to \mathcal M_n(\mathbb C)$, that is, 
 $$K(\boldsymbol z,\boldsymbol w)= c(k) K(k^{-1}\cdot \boldsymbol z,k^{-1} \cdot \boldsymbol w)c(k)^*,~k\in \mathbb K,~ \boldsymbol z, \boldsymbol w\in\Omega,$$ 
 where $c$ is an irreducible unitary representation.
 If  the kernel $K$ is also invariant under $\mathbb K$, it follows that, $K(\boldsymbol z, \boldsymbol w)= c(k)K(\boldsymbol z, \boldsymbol w)c(k)^*$, that is, $K(\boldsymbol z, \boldsymbol w)c(k)=c(k)K(\boldsymbol z, \boldsymbol w)$ for all $k\in \mathbb K$. By Schur's Lemma, $K(\boldsymbol z, \boldsymbol w)=\kappa(\boldsymbol z, \boldsymbol w) I_n$ for some scalar $\kappa(\boldsymbol z, \boldsymbol w)$. The kernel  $K(\boldsymbol z, \boldsymbol w)$ is non-negative definite, therefore  $\kappa(\boldsymbol z, \boldsymbol w)$ is non-negative definite also. Moreover, since $K(\boldsymbol z, \boldsymbol w)$ is invariant under $\mathbb K$, it follows that $\kappa(\boldsymbol z, \boldsymbol w)$ is invariant under $\mathbb K$ as well. This completes the proof.
\end{proof} 
\begin{rem}
As we have pointed out earlier, under some additional assumptions, any scalar-valued non-negative definite kernel $K$ on $\Omega\times\Omega$ quasi-invariant under $\mathbb K$ can be shown to be of the form 
$\sum_{\underline{s} \in \vec{\mathbb Z}^r_+} a_{\underline{s}}E_{\underline{s}}$ for some sequence $\{a_{\underline{s}}\}_{ \s \in \vec{\mathbb Z}^r_+}$ of non-negative real numbers.
\end{rem}
\section{A class of quasi-invariant kernels}
Let $(\mathcal P, \langle \cdot, \cdot \rangle_{\mathcal F})$ denote the linear space of all polynomials in $d$-variables equipped with the Fischer-Fock inner product and let $(\mathbb C^d, \inp{\cdot}{\cdot}_2)$ denote the Euclidean inner product space. We have 
$$ (\mathbb C^d,\inp{\cdot}{\cdot}_2) \otimes (\mathcal P, \langle \cdot, \cdot \rangle_{\mathcal F}) = \bigoplus_{\ell=0}^\infty  (\mathbb C^d \otimes \mathcal P_\ell, \inp{\cdot}{\cdot}_{\mathcal F_\ell}),$$ where 
the linear space $(\mathbb C^d \otimes \mathcal P_\ell, \inp{\cdot}{\cdot}_{\mathcal F_\ell})$ denotes the subspace of $(\mathbb C^d, \inp{\cdot}{\cdot}_2) \otimes (\mathcal P, \langle \cdot, \cdot \rangle_{\mathcal F})$ consisting of homogeneous polynomials of degree $\ell$. Thus the reproducing kernel of $(\mathbb C^d \otimes \mathcal P_\ell, \inp{\cdot}{\cdot}_{\mathcal F_\ell})$ is of the form $\frac{\inp{z}{w}^\ell}{\ell!}I_d$.

Recall  that the unitary group $\mathcal U(d)$ acts on $\mathcal P$ by 
$(\pi(u)(p))(\boldsymbol z) =  p(u^{-1} \cdot \boldsymbol z)$, $ p\in \mathcal P$. 
Therefore, the map 
given by the formula:
\begin{equation}\label{actionu1}
\big(\hat{\pi}(u)(p)\big) (\boldsymbol z) := u (p (u^{-1} \cdot \boldsymbol z))
,\,\, p\in \mathbb C^d \otimes \mathcal P,\,\, u\in\mathcal U(d)\end{equation}
is an unitary homomorphism. Let $\pi_\ell(u)$ denote the restriction of $\pi(u)$ to $\mathcal P_\ell$ and $\hat{\pi}_\ell(u)$ be the restriction of $\hat{\pi}(u)$ to $ \mathbb C^d \otimes \mathcal P_\ell$.
Evidently, the subspaces $ \mathbb C^d \otimes \mathcal P_\ell $, $\ell \in \mathbb Z_+$, are not only invariant under the action $\hat{\pi}$ of $\mathcal U(d)$ but also the restriction of $\hat{\pi}_\ell$ to these subspaces is unitary.


There is a second action $\tilde{\pi}$ of the unitary group $\mathcal U(d)$ on $\mathbb C^d \otimes \mathcal P$  given by the formula:
\begin{equation} \label{actionu2}
\big (\tilde{\pi}(u)(p) \big ) (\boldsymbol z) = 
\overline{u} ({p}(u^{-1} \cdot \boldsymbol z) )
,\,\, p\in \mathbb C^d \otimes \mathcal P.\end{equation}
Like before, the restriction $\tilde{\pi}_\ell(u)$ of $\tilde{\pi}(u)$ to the space $ \mathbb C^d\otimes \mathcal P_\ell $ is unitary.

\subsection{Decomposition of $\tilde{\pi}_\ell$} 

Let $ \boldsymbol A =(A_1,\ldots,A_n)$ be an $n$-tuple of bounded linear operators (not necessarily commuting) $A_i:\mathcal H_1 \to \mathcal H_2$, $1\leq i \leq n$, where the Hilbert space $\mathcal H_1$ is possibly different from $\mathcal H_2$. The operators $D_{\boldsymbol A}:\mathcal H_1\to \mathcal H_2\oplus\cdots\oplus \mathcal H_2$ and $D^{\boldsymbol A}:\mathcal H_1\oplus\cdots\oplus\mathcal H_1\to\mathcal H_2$ are defined by the rule 
\begin{align*}
    D_{\boldsymbol A}(h)=(A_1h,\ldots,A_nh),~h\in \mathcal H_1~~~\mbox{and}\\
    D^{\boldsymbol A} \left(\begin{smallmatrix}h_1\\
\vdots\\
h_n
\end{smallmatrix}\right)=A_1h_1+\cdots+ A_n h_n,~h_1,\ldots,
h_n\in\mathcal H_1.
\end{align*}
It is easy to verify that 
    $(D^{\boldsymbol A})^*=D_{\boldsymbol A^*}$.

For any $u\in \mathcal U(d)$, $f= \left(\begin{smallmatrix}f_1\\
\vdots\\
f_d
\end{smallmatrix}\right)\in \mathbb C^d\otimes \mathcal P_{\ell}$ and $\boldsymbol z\in \mathbb C^d$, we have
\begin{align*}
    \sum_{i=1}^d z_i (u^{\dagger}(f\circ u))_i(\boldsymbol z)=\inp{u^\dagger (f\circ u)(\boldsymbol z)}{\overline{\boldsymbol z}}_{\mathbb C^d}=\inp{ (f\circ u)(\boldsymbol z)}{\overline{u\cdot \boldsymbol z}}_{\mathbb C^d}
=\sum_{i=1}^d (u\cdot \boldsymbol z)_i f_i(u\cdot \boldsymbol z).
\end{align*}
Thus, $\tilde{\pi}$ leaves the subspace $\tilde{\mathcal V}_\ell \subseteq  (\mathbb C^d\otimes \mathcal P_{\ell}, \inp{\cdot}{\cdot}_{\mathcal F_\ell})$ invariant, where 
$$\tilde{\mathcal V}_\ell=\left\{\left(\begin{smallmatrix}f_1\\
\vdots\\
f_d
\end{smallmatrix}\right) \in \mathbb C^d \otimes \mathcal P_\ell : z_1f_1+\cdots+z_df_d=0\right\}.$$ 
We claim that the subspace $\tilde{\mathcal V}_\ell^\perp \subseteq (\mathbb C^d\otimes \mathcal P_{\ell}, \inp{\cdot}{\cdot}_{\mathcal F_\ell})$ is $\Big \{ \left (\begin{smallmatrix} \partial_1 g \\ \vdots \\ \partial_d g \end{smallmatrix} \right ) : g \in \mathcal P_{\ell+1}\Big \}.$ 

 To verify the claim, let $M_{z_i}^{(\ell)}:\mathcal P_{\ell}\to \mathcal P_{\ell+1}$ be the linear map 
       $M_{z_i}^{(\ell)}(p)=z_ip,~p\in \mathcal P_{\ell}.$ Setting $\boldsymbol M^{(\ell)} = (M_{z_1}^{(\ell)},\ldots,M_{z_d}^{(\ell)})$, we have
    $\tilde{\mathcal V}_{\ell}=\ker  D^{\boldsymbol M^{(\ell)}}.$
     Thus $\tilde{\mathcal V}_{\ell}^\perp={\rm ran~}( D^{\boldsymbol M^{(\ell)}})^*=
     {\rm ran~} D_{{\boldsymbol M^{(\ell)}}^*}.$ From the identity $\inp{p}{z_iq}_{\mathcal F}=\inp{\partial_i p}{q}_{\mathcal F}$ for any pair of polynomials proved in \cite{upmeier}, Proposition 4.11.36, it follows that $M_{z_i}^{(\ell)^*}=\partial_i$ completing the verification of the claim.


\begin{lemma}\label{first decomposition} The reproducing kernel $\tilde{K}_\ell$ of the inner product space $\tilde{\mathcal V}_\ell$ is given by the formula: 
$$\tilde{K}_\ell(\boldsymbol z, \boldsymbol w)=\frac{\ell}{(\ell+1)\ell!}\inp{\boldsymbol z}{\boldsymbol w}^{\ell-1}\left(\inp{\boldsymbol z}{\boldsymbol w}I_d-\overline{\boldsymbol w}\boldsymbol z^\dagger\right).$$

The reproducing kernel $\tilde{K}_\ell^\perp$ of $\tilde{\mathcal V}_\ell^\perp$ is given by the formula: 
$$\tilde{K}_\ell^\perp(\boldsymbol z, \boldsymbol w) = \frac{\inp{\boldsymbol z}{\boldsymbol w}^{\ell-1}}{(\ell+1) \ell!}\left(\inp{\boldsymbol{z}}{\boldsymbol{w}}I_d + \overline{\boldsymbol w} \boldsymbol{z}^\dagger\right).$$ 

Here, $\overline{\boldsymbol w} \boldsymbol z^\dagger$ is the matrix product of the column vector $\overline{\boldsymbol w}$ and the row vector  $\boldsymbol z^\dagger$.
\end{lemma} 

\begin{proof}
Let $\boldsymbol \zeta=(\zeta_1,\ldots,\zeta_d)$ be an arbitrary vector in $\mathbb C^d$. 
First note that 
\begin{align*}
\sum_{i=1}^d z_i \inp{\tilde{K}_\ell(\boldsymbol z,\boldsymbol w)\boldsymbol \zeta}{\boldsymbol e_i}
& =\frac{\ell}{(\ell+1)\ell!}\inp{\boldsymbol z}{\boldsymbol w}^{\ell-1}\left(\sum_{i=1}^d z_i \big \langle \inp{\boldsymbol z}{\boldsymbol w}\boldsymbol \zeta-\overline{\boldsymbol w}\inp{\boldsymbol z}{\bar{\boldsymbol \zeta} } , \boldsymbol e_i \big \rangle \right) \\
&=\frac{\ell}{(\ell+1)\ell!}\inp{\boldsymbol z}{\boldsymbol w}^{\ell-1}\sum_{i=1}^d \left( \inp{\boldsymbol z}{\boldsymbol w} z_i\zeta_i - z_i \bar{w}_i \inp{\boldsymbol z}{\bar{\boldsymbol \zeta} }\right)\\ 
&= \frac{\ell}{(\ell+1)\ell!}\inp{\boldsymbol z}{\boldsymbol w}^{\ell-1} \big  (\inp{\boldsymbol z}{\boldsymbol w} \inp{\boldsymbol z}{\bar{\boldsymbol \zeta} } -   \inp{\boldsymbol z}{\boldsymbol w} \inp{\boldsymbol z}{\bar{\boldsymbol \zeta} } \big )\\
&= 0.
\end{align*}
It follows that the vector  $\tilde{K}_\ell(\cdot,\boldsymbol w)\boldsymbol \zeta\in \tilde{\mathcal V}_\ell$. In order to complete the proof of the first part it suffices to show that for any
$f$ in  $\tilde{\mathcal V}_\ell$, $\boldsymbol w, \boldsymbol \zeta\in \mathbb C^d$, and $i=1,\ldots,d$ $\inp{f}{\tilde {K}_\ell(\cdot,\boldsymbol w)\boldsymbol e_i}_{\mathcal F_{\ell}}=\inp{f(\boldsymbol w)}{\boldsymbol e_i}_{\mathbb C^d}$. Note that
\begin{align*}
    \inp{f}{\inp{\boldsymbol z}{\boldsymbol w}^{\ell-1}\overline{\boldsymbol w}\boldsymbol z^\dagger \boldsymbol e_i}_{\mathcal F_{\ell}}&=\sum_{j=1}^d \inp{f_j}{\inp{\boldsymbol z}{\boldsymbol w}^{\ell-1}\overline{w}_jz_i}_{\mathcal F}
    \\&
    =\sum_{j=1}^d w_j\inp{\partial_if_j}{\inp{\boldsymbol z}{\boldsymbol w}^{\ell-1}}_{\mathcal F}\\&
    =(\ell-1)!\sum_{j=1}^d w_j(\partial_if_j)(\boldsymbol w)\\ &
    =(\ell-1)!\Big(\partial_i\Big(\sum_{j=1}^d z_jf_j\Big)(\boldsymbol w)-f_i(\boldsymbol w)\Big)\\&
    =-(\ell-1)!f_i(\boldsymbol w).
    \end{align*}
    Hence we have
    \begin{equation}\label{eqn kerenl of V}
        \inp{f}{\inp{\boldsymbol z}{\boldsymbol w}^{\ell-1}\overline{\boldsymbol w}\boldsymbol z^\dagger \boldsymbol e_i}_{\mathcal F_{\ell}}=-(\ell-1)!
    \inp{f(\boldsymbol w)}{\boldsymbol e_i}_{\mathbb C^d}.
    \end{equation}
    Here the second equality follows since  $\inp{p}{z_iq}_{\mathcal F}=\inp{\partial_i p}{q}_{\mathcal F}$ for any pair of polynomials  $p,q$ (see \cite[Proposition 4.11.36]{upmeier}), and the third equality from the reproducing property of the kernel function of $\mathcal P_{\ell-1}$.
    Now, using \eqref{eqn kerenl of V}, we see that
    \begin{align*}
       \inp{f}{\tilde{K}_\ell(\cdot,\boldsymbol w)\boldsymbol e_i}_{\mathcal F_{\ell}}&=\frac{\ell}{(\ell+1)\ell!}\inp{f}{\inp{\boldsymbol z}{\boldsymbol w}^{\ell-1}\left(\inp{\boldsymbol z}{\boldsymbol w}e_i-\overline{\boldsymbol w}\boldsymbol z^\dagger \boldsymbol e_i\right)}_{\mathcal F_{\ell}}\\
       &=\frac{\ell}{(\ell+1)}\Big(1+\frac{1}{\ell}\Big)\inp{f(\boldsymbol w)}{\boldsymbol e_i}\\&=
       \inp{f(\boldsymbol w)}{\boldsymbol e_i}.
       \end{align*}
       This verifies the formula for $\tilde{K}_\ell$.  

We note that the reproducing kernel of $(\mathbb C^d \otimes \mathcal P_\ell, \inp{\cdot}{\cdot}_{\mathcal F_\ell})$ is  $\frac{\inp{\boldsymbol z}{\boldsymbol w}^\ell}{\ell!}I_d$. Now, the verification of the formula for $\tilde{K}_\ell^\perp$ follows from part (1) and the equality: $$\frac{\inp{\boldsymbol z}{\boldsymbol w}^\ell}{\ell!}I_d = \tilde{K}_\ell(\boldsymbol z, \boldsymbol w) + \tilde{K}_\ell^\perp(\boldsymbol z, \boldsymbol w),$$ which follows from general theory of reproducing kernel Hilbert spaces. 
\end{proof}


The proof of Proposition \ref{qK1} giving an explicit description of a quasi-invariant kernel under $\mathcal U(d)$ transforming as in Definition \eqref{qinv} with $c(u)= \bar{u}$ is facilitated by the set of three lemmas proved below. 
\begin{lemma}\label{lemschur}
Let $A$ be a $d\times d$ complex matrix such that $uA=Au$ for all unitary matrices $u$ with $u(\boldsymbol e_1)=\boldsymbol e_1$. Then $A$ is of the form $ \left (\begin{smallmatrix} a_1 & 0\\ 0& a_2 I_{d-1} \end{smallmatrix}  \right )$ for some complex numbers $a_1$ and $a_2$.
\end{lemma}
\begin{proof} Let $A= \left (\begin{smallmatrix} A_1 & A_3^\dagger\\ A_4& ~A_2 \end{smallmatrix}  \right )$, where $A_3$ and $A_4$ are column vectors in $\mathbb C^{d-1}$ and $A_2$ is in $\mathcal M_{d-1}(\mathbb C)$. By hypothesis, we get $A_3=A_4=0$ and $vA_2=A_2v$ for all $v\in \mathcal U(d-1).$ Now the conclusion follows by an application of the Schur's lemma.
\end{proof}

\begin{lemma}\label{lemdiagonal}
Suppose that $K:\mathbb B_d \times \mathbb B_d \to \mathcal M_n(\mathbb C)$ is a sesqui-analytic Hermitian function satisfying the  rule $K(\lambda \cdot \boldsymbol z,\lambda\cdot \boldsymbol w)=K(\boldsymbol z,\boldsymbol w)$ for all $\lambda$ on the unit circle $\mathbb T$. Then $K(\boldsymbol z,\boldsymbol w)$ is of the form
$$\sum_{\ell=0}^\infty\sum_{\substack{\alpha,\beta\in \mathbb Z_+^d\\|\alpha|=|\beta|=\ell}}A_{\alpha, \beta}{\boldsymbol z}^\alpha\overline{{\boldsymbol w}}^\beta, \boldsymbol z, \boldsymbol w\in \mathbb B_d,$$ where $A_{\alpha,\beta}$ are $n\times n$ complex matrices.
\end{lemma}
\begin{proof}
Let $K(\boldsymbol z, \boldsymbol w)=\sum_{\alpha,\beta\in \mathbb Z^d_+} A_{\alpha,\beta}{\boldsymbol z}^{\alpha}\overline{\boldsymbol w}^{\beta}$, $\boldsymbol z, ,\boldsymbol w\in \mathbb B_d.$ By hypothesis, we have
$$\sum_{\alpha,\beta\in \mathbb Z^d_+} A_{\alpha,\beta}{\boldsymbol z}^{\alpha}\overline{\boldsymbol w}^{\beta}=\sum_{\alpha,\beta\in \mathbb Z^d_+} A_{\alpha,\beta}{\lambda}^{|\alpha|-|\beta|}{\boldsymbol z}^{\alpha}\overline{\boldsymbol w}^{\beta},~ \boldsymbol z, ,\boldsymbol w\in \mathbb B_d,~\lambda\in \mathbb T.$$
Comparing coefficients in both sides, we get $A_{\alpha,\beta}(1-\lambda^{|\alpha|-|\beta|})=0$ for all $\lambda\in \mathbb T$. Hence it follows that $A_{\alpha,\beta}=0$ if $|\alpha|\not=|\beta|.$ This completes the proof.
\end{proof}
For any $\boldsymbol z \in \mathbb B_d$, $\|\boldsymbol z\| = r$, there is a $u_{\boldsymbol z}\in \mathcal U(d)$ with the property: 
$u_{\boldsymbol z}(\boldsymbol z)= r \boldsymbol e_1.$ The unitary $u_{\boldsymbol z}$ can be determined explicitly, namely, $u_{\boldsymbol z}^* =\left ( \begin{smallmatrix} \tfrac{\boldsymbol z}{r} | & \bigstar  \end{smallmatrix} \right )$, where $\boldsymbol z$ is the column vector with components $z_1, \ldots , z_d$. For any choice of two sets of complex numbers, $\{a_{m,1}: m \in \mathbb Z_+\}$ and $\{a_{m,2}: m \in \mathbb Z_+\}$ with $a_{0,1} = a_{0,2}$, set  
$$D_i(r,r) := \sum_{m=0}^\infty a_{m,i}r^{2m}, r\in [0,1), i=1,2.$$
Also, for any fixed $z\in \mathbb B_d$ with $\|z\|=r$, let $\mathcal U_{\boldsymbol z}$ be the set $\{u_{\boldsymbol z} \in \mathcal U(d): u_{\boldsymbol z}(\boldsymbol z) = \|\boldsymbol z\| \boldsymbol e_1\}$.

\begin{lemma} \label{lem 4.10}For any $u_{\boldsymbol z}\in \mathcal U_{\boldsymbol z}$,   we have 
$$ u_{\boldsymbol z}^\dagger \left (\begin{matrix} D_1(r,r) & 0\\ 0& D_2(r,r) I_{d-1} \end{matrix}  \right ) \overline{u_{\boldsymbol z}}=\big(D_1(r,r) - D_2(r,r)\big)\frac{\overline{\boldsymbol z}  \boldsymbol z^\dagger}{r^2} + D_2(r,r) I_d.$$
\end{lemma}
\begin{proof} For any $u_{\boldsymbol z}\in \mathcal U_{\boldsymbol z}$, we have 
\begin{align*}
 u_{\boldsymbol z}^\dagger \left (\begin{smallmatrix} D_1(r,r) & 0\\ 0& D_2(r,r) I_{d-1} \end{smallmatrix}  \right ) \overline{u_{\boldsymbol z}}
&= u_{\boldsymbol z}^\dagger \left (\begin{smallmatrix} D_1(r,r)-D_2(r,r) & 0\\ 0& 0\end{smallmatrix}  \right )\overline{u_{\boldsymbol z}}+u_{\boldsymbol z}^\dagger D_2(r,r) I_{d}  \overline{u_{\boldsymbol z}}\\
&= D_1(r,r)-D_2(r,r) u_{\boldsymbol z}^\dagger E_{11}\overline{u_{\boldsymbol z}}+ D_2(r,r) I_{d} u_{\boldsymbol z}^\dagger \overline{u_{\boldsymbol z}}\\
\end{align*}
Since $u_{\boldsymbol{z}}(\boldsymbol{z})=\|\boldsymbol{z}\| e_1$, we get that $u_{\boldsymbol{z}}^\dagger e_1 = \frac{\bar{z}}{r}$. Thus, \[u_{\boldsymbol{z}}^\dagger E_{11}\overline{u_{\boldsymbol{z}}} = u_{\boldsymbol{z}}^\dagger e_1 \overline{e_1}^\dagger \overline{u_{\boldsymbol{z}}} = \frac{\bar{\boldsymbol{z}} \boldsymbol z^\dagger}{r^2}.\] 
This completes the proof. 
\end{proof}
\begin{rem}
An unitary $u_{\boldsymbol z}\in \mathcal U_{\boldsymbol z}$ such that  $u_{\boldsymbol{z}}(\boldsymbol{z})=\|\boldsymbol{z}\| e_1$ is not uniquely determined. However, if $\boldsymbol z\not =0$, 
we see that 
$$u_{\boldsymbol z}^\dagger \left (\begin{matrix} D_1(r,r) & 0\\ 0& D_2(r,r) I_{d-1} \end{matrix}  \right ) \overline{u_{\boldsymbol z}}$$
is independent of the choice of $u_{\boldsymbol z}$ by Lemma \ref{lem 4.10}. 
\end{rem}
\begin{prop} \label{qK1} Suppose that $K:\mathbb B_d \times \mathbb B_d \to \mathcal M_d(\mathbb C)$ is a sesqui-analytic 
Hermitian function satisfying the transformation rule with the multiplier  $c(u) = \overline{u}$: 
\begin{equation}\label{eqnquasiinv}
u^\dagger K(u\cdot \boldsymbol z, u\cdot \boldsymbol w) \overline{u} = K(z,w), u\in\mathcal U(d).
\tag{$*$}\end{equation}
Then  $K$ must be of the form 
\begin{align*}
K(\boldsymbol z, \boldsymbol z) &= u_{\boldsymbol z}^\dagger \left (\begin{smallmatrix} {D}_1(r,r) & 0\\ 0& {D}_2(r,r) I_{d-1} \end{smallmatrix}  \right ) \overline{u_{\boldsymbol z}},\, u_{\boldsymbol z} \in \mathcal U_{\boldsymbol z}, 
\end{align*}
where ${D}_i(r,r)$, $i=1,2$ are real analytic function on $[0,1)$ of the form $\sum_{m=0}^\infty {a}_{m,i}r^{2m}$ with ${a}_{0,1}={a}_{0,2}$. 
\end{prop} 
\begin{proof}
Note that $u^\dagger K(0,0) \bar{u} = K(0,0)$ implying $K(0,0)$ must be a scalar times $I_d$. 
Let $\boldsymbol z\in \mathbb B_d$ and $\boldsymbol z\neq 0.$ Putting $w=z$ and $u=u_z\in \mathcal U_{\boldsymbol z}$ in \eqref{eqnquasiinv} we get that
\begin{align}\label{eqn4.5}
 K(\boldsymbol z, \boldsymbol z)&=   u_{\boldsymbol z}^\dagger K(u_{\boldsymbol z}(\boldsymbol z), u_{\boldsymbol z}(\boldsymbol z)) \overline{u_{\boldsymbol z}}\nonumber\\
 &=u_{\boldsymbol z}^\dagger K(\|\boldsymbol z\|\boldsymbol e_1, \|\boldsymbol z\|\boldsymbol e_1) \overline{u_{\boldsymbol z}}.
\end{align}
Using this expression of $K(z,z)$ in \eqref{eqnquasiinv} we see that
\begin{equation}
    u_{\boldsymbol z}^\dagger K(\|\boldsymbol z\|\boldsymbol e_1, \|\boldsymbol z\|\boldsymbol e_1) \overline{u_{\boldsymbol z}}=u^\dagger u_{ {u\cdot \boldsymbol z}}^\dagger K(\|u\cdot \boldsymbol z\|\boldsymbol e_1, \|u\cdot \boldsymbol z\|\boldsymbol e_1) \overline{u_{{u\cdot\boldsymbol z}}}~ \overline{u}.
\end{equation}
Equivalently, we have
\begin{equation}\label{eqnquasiinv1}
    \overline{u_{{u\cdot\boldsymbol z}}}~ \overline{u}u_{\boldsymbol z}^\dagger K(\|\boldsymbol z\|\boldsymbol e_1, \|\boldsymbol z\|\boldsymbol e_1) =K(\|\boldsymbol z\|\boldsymbol e_1, \|\boldsymbol z\|\boldsymbol e_1)  \overline{u_{{u\cdot\boldsymbol z}}}~ \overline{u}u_{\boldsymbol z}^\dagger,~\mbox{for all}~u\in \mathcal U(d), u_{\boldsymbol z}\in \mathcal U_{\boldsymbol z}.
\end{equation}
Note that $\overline{u_{{u\cdot\boldsymbol z}}}~ \overline{u}u_{\boldsymbol z}^\dagger$ is a unitary and $$\overline{u_{{u\cdot\boldsymbol z}}}~ \overline{u}u_{\boldsymbol z}^\dagger(\boldsymbol e_1)=\overline{u_{{u\cdot\boldsymbol z}}}~ \overline{u}(\frac{\overline{z}}{\|\boldsymbol z\|})=\overline{\frac{u_{u\cdot\boldsymbol z}(u\cdot \boldsymbol z)}{\|\boldsymbol z\|}}=\boldsymbol e_1.$$
Moreover, if $v$ is a unitary in $\mathcal U(d)$
with $v(\boldsymbol e_1)=\boldsymbol e_1$, then $v$ can be written as  
$\overline{u_1}~\overline{u}u_2^\dagger$, where $u=\overline{v}u_{\boldsymbol z}$, $u_2=u_{\boldsymbol z}$ and $u_1=I_d$. Since $\overline{v}u_{\boldsymbol z}(\boldsymbol z)=\|\boldsymbol z\|\overline{v}(\boldsymbol e_1)=\|\boldsymbol z\|\boldsymbol e_1$, we see that $u_1=I_d\in \mathcal U_{u\cdot\boldsymbol z}$. Consequently, it follows that the set $\{\overline{u_{{u\cdot\boldsymbol z}}}~ \overline{u}u_{\boldsymbol z}^\dagger:u\in \mathcal U(d), u_{\boldsymbol z}\in \mathcal U_{\boldsymbol z},  u_{u\cdot \boldsymbol z}\in \mathcal U_{u\cdot \boldsymbol z} \}$ coincides with the set $\{v\in \mathcal U(d):v(\boldsymbol e_1)=\boldsymbol e_1\}$.
This together with \eqref{eqnquasiinv1} gives
\begin{equation}\label{eqnquasiinv2}
    v K(\|\boldsymbol z\|\boldsymbol e_1, \|\boldsymbol z\|\boldsymbol e_1) =K(\|\boldsymbol z\|\boldsymbol e_1, \|\boldsymbol z\|\boldsymbol e_1)v,
\end{equation}
for all $v\in \mathcal U(d)$ with $v(\boldsymbol e_1)=\boldsymbol e_1.$ Hence by Lemma \ref{lemschur} we get that 
\begin{equation}\label{eqn4.8}
    K(\|\boldsymbol z\|\boldsymbol e_1, \|\boldsymbol z\|\boldsymbol e_1)= \left (\begin{smallmatrix} K_1(\|\boldsymbol z\|\boldsymbol e_1, \|\boldsymbol z\|\boldsymbol e_1) & 0\\ 0& K_2(\|\boldsymbol z\|\boldsymbol e_1, \|\boldsymbol z\|\boldsymbol e_1) I_{d-1} \end{smallmatrix}  \right ),
\end{equation}
where $K_1$ and $K_2$ are two scalar-valued sesqui-analytic Hermitian functions on $\mathbb B_d\times \mathbb B_d.$ Applying Lemma \ref{lemdiagonal}, we 
infer that 
$$K(\boldsymbol z, \boldsymbol z)=\sum_{\ell=0}^\infty\sum_{|\alpha|=|\beta|=\ell}a_{\alpha,\beta} \boldsymbol z^{\alpha}\bar{\boldsymbol z}^{\beta},\,\, a_{\alpha,\beta}\in \mathcal M_d(\mathbb C).$$ 
Consequently, we have the equality
\begin{equation}\label{eq:4.10a}
K(\|\boldsymbol z\|\boldsymbol e_1, \|\boldsymbol z\|\boldsymbol e_1)=\sum_{\ell=0}^\infty a_{\ell \varepsilon_1, \ell \varepsilon_1} \|\boldsymbol z\|^{2\ell}.\end{equation}
Combining Equation \eqref{eq:4.10a} with the Equations \eqref{eqn4.5} and \eqref{eqn4.8}, completes the verification of the first of the two equalities claimed for the kernel $K$. 
\end{proof}
Now, we obtain a characterization of the non-negative definite quasi-invariant kernels. 
\begin{theorem} \label{thm positive definiteness}
Any sesqui-analytic Hermitian function quasi-invariant with multiplier $c(u)=\bar{u}$ is of the form 
\[\tilde{K}^{(\alpha, \beta)}(\boldsymbol z, \boldsymbol w) = \sum_{j=1}^\infty \alpha_j \tilde{K}_j(\boldsymbol z, \boldsymbol w)+  \sum_{j=0}^\infty \beta_j \tilde{K}_j^\perp(\boldsymbol z, \boldsymbol w),\,\, \alpha_j, \beta_j \in \mathbb C.\]

\end{theorem}

\begin{proof}
First, since any sesqui-analytic Hermitian function quasi-invariant with multiplier $c(u)=\bar{u}$, it  must be of the form prescribed in Proposition \ref{qK1}.  Applying Lemma \ref{lem 4.10} to it, and then polarizing the result, we see that it must be of the form 
\begin{align} \label{kernelform1} \tilde{K}(\boldsymbol z, \boldsymbol w)
= \sum_{\ell=1}^\infty\big(a_{\ell,1} - a_{\ell,2}\big)\inp{\boldsymbol z}{\boldsymbol w}^{\ell-1}\overline{\boldsymbol w}  \boldsymbol z^\dagger +\sum_{\ell=0}^\infty a_{\ell,2}\inp{\boldsymbol z}{\boldsymbol w}^{\ell} I_d, \; \; \boldsymbol z, \boldsymbol w\in \mathbb B_d,
\tag{$\sharp$}\end{align}  
for some choice of complex numbers $a_{\ell,1}$,  $\ell \in \mathbb N$, and $a_{\ell,2}$,  $\ell \in \mathbb Z_+$. For any $\ell\geq 1$, by Lemma \ref{first decomposition}, we have
\begin{align*}
&\big(a_{\ell,1} - a_{\ell,2}\big)\inp{\boldsymbol z}{\boldsymbol w}^{l-1}\overline{\boldsymbol w} \boldsymbol z^\dagger + a_{\ell,2}\inp{\boldsymbol z}{\boldsymbol w}^{l} I_d\\
&\phantom{AAAAAAAAAAAA}=a_{\ell,1}\inp{\boldsymbol z}{\boldsymbol w}^{\ell} I_d-\big(a_{\ell,2} - a_{\ell,1}\big)\inp{\boldsymbol z}{\boldsymbol w}^{\ell-1}\big (\inp{\boldsymbol z}{\boldsymbol w} -\overline{\boldsymbol w}  \boldsymbol z^\dagger \big )\\
&\phantom{AAAAAAAAAAAA} =a_{\ell,1} \ell! \big(\tilde{K}_\ell + \tilde{K}_\ell^\perp \big)-\big(a_{\ell,1} - a_{\ell,2}\big)\frac{(\ell+1)\ell!}{\ell} \tilde{K}_\ell\\
&\phantom{AAAAAAAAAAAA}=a_{\ell,1} \ell! \tilde{K}_\ell^\perp +\big((\ell+1)a_{\ell,2}-a_{\ell,1}\big)(\ell-1)!\tilde{K}_\ell.
\end{align*}
Thus, we have  
\begin{align*}
\tilde{K}(z,w)&=a_{0,2} I_d+\sum_{\ell=1}^\infty  \big (a_{\ell,1} \ell! \tilde{K}_\ell^\perp +\big((\ell+1)a_{\ell,2}-a_{\ell,1}\big)(\ell-1)!\tilde{K}_\ell \big)\\
&= \sum_{j=1}^\infty \alpha_j \tilde{K}_j(\boldsymbol z, \boldsymbol w)+  \sum_{j=0}^\infty \beta_j \tilde{K}_j^\perp(\boldsymbol z, \boldsymbol w),\end{align*}
where $\alpha_j = \big((j+1)a_{j,2}-a_{j,1}\big)(j-1)!$, $\beta_j = a_{j,1} j!$.

\end{proof}

\subsection{Decomposition of $\hat{\pi}$} Consider the two subspaces $\widehat{\mathcal V}_\ell$ and $\widehat{\mathcal W}_\ell$ of 
$(\mathbb C^d\otimes \mathcal P_{\ell}, \inp{\cdot}{\cdot}_{\mathcal F_\ell})$:
$$\widehat{\mathcal V}_\ell=\left\{f:=\left(\begin{smallmatrix}f_1\\
\vdots\\
f_d
\end{smallmatrix}\right) \in \mathbb C^d \otimes \mathcal P_\ell  :  \partial_1f_1+\cdots+\partial_d f_d=0\right\}$$ and
$$\widehat{\mathcal W}_{\ell}=\Big \{ \left (\begin{smallmatrix} z_1 g \\ \vdots \\ z_d g \end{smallmatrix} \right ) : g \in \mathcal P_{\ell-1}\Big \}.$$
Evidently, the subspace $\widehat{\mathcal W}_\ell$ is invariant under the unitary representation $\hat{\pi}$. Also, we check that the subspace $\widehat{\mathcal V}_\ell^\perp \subseteq(\mathbb C^d\otimes \mathcal P_{\ell}, \inp{\cdot}{\cdot}_{\mathcal F_\ell})$ is $\widehat{\mathcal W}_{\ell}.$

To verify this, let $M_{z_i}^{(\ell)}:\mathcal P_{\ell-1}\to \mathcal P_{\ell}$ be the linear map 
       $M_{z_i}^{(\ell)}(p)=z_ip,~p\in \mathcal P_{\ell}.$
Clearly, setting $\boldsymbol M^{(\ell)} = (M_{z_1}^{(\ell)},\ldots,M_{z_d}^{(\ell)})$, we see that $\widehat{\mathcal W}_{\ell}={\rm ran~}( D^{\boldsymbol M^{(\ell)}})$.  
Note that for any $\alpha,\beta \in \mathbb Z^d_+$, 
$\langle \boldsymbol z^{\alpha+\varepsilon_i}, \boldsymbol z^{\beta}\rangle_{\mathcal F}={\beta}!  \delta_{\alpha+\varepsilon_i,\beta}$. Thus we have
$$\langle z_i p,q\rangle_{\mathcal F}=\langle p, \partial_i q\rangle_{\mathcal F},~ p,q\in \mathcal P.$$
Hence it follows that   $M_{z_i}^{(\ell)^*}=\partial_i$. Therefore 
$\widehat{\mathcal V}_{\ell}=\ker  D^{\boldsymbol M^{(\ell)^*}}.$
 Since  $\big (D^{\boldsymbol M^{(\ell)}}\big )^*=D_{\boldsymbol M^{(\ell)^*}}$, we conclude that 
 $$\widehat{\mathcal V}_{\ell}^\perp={\rm ran~} D_{\boldsymbol M^{(\ell)}}=\hat{\mathcal W}_{\ell}.$$

     Therefore, $\widehat{\mathcal V}_\ell$ is also invariant under the representation $\hat{\pi}$. 

\begin{lemma} \label{lem:3.7} Consider the inner product space $(\mathbb C^d \otimes \mathcal P_{\ell}, \inp{\cdot}{\cdot}_{\mathcal F_\ell})$. Then 
\begin{enumerate}
\item The reproducing kernel $\widehat{K}_\ell$ of $\widehat{\mathcal V}_\ell$ is $$\widehat{K}_\ell(\boldsymbol z, \boldsymbol w):=\frac{1}{(\ell+d-1)(\ell-1)!}\inp{\boldsymbol z}{\boldsymbol w}^{\ell-1}\left(\frac{(\ell+d-1)}{\ell}\inp{\boldsymbol z}{\boldsymbol w}I_d-\boldsymbol z\overline{\boldsymbol w}^\dagger\right),$$

where $ \boldsymbol z \overline{\boldsymbol w}^\dagger$ is the matrix product of the column vector $\boldsymbol z$ and the row vector $\overline{\boldsymbol w}^\dagger$.

\item The reproducing kernel $\widehat{K}_\ell^\perp$ of $\widehat{\mathcal V}_\ell^\perp$ is $\frac{1}{(\ell+d-1)(\ell-1)!}\inp{\boldsymbol z}{\boldsymbol w}^{\ell-1}\boldsymbol z\overline{\boldsymbol w}^\dagger.$ 
\end{enumerate} 
\end{lemma}

\begin{proof}
Clearly, part (2) is a direct consequence of part (1) of the Lemma. Therefore, we will prove only part (1), which is similar to the proof of part (1) of Lemma \ref{first decomposition}. 
Let $\boldsymbol \zeta=(\zeta_1,\ldots,\zeta_d)$ be any vector in $\mathbb C^d$. As before, we  note that 
\begin{align*}
\inp{\widehat{K}_\ell(\boldsymbol z,\boldsymbol w)\boldsymbol \zeta}{\boldsymbol e_j}
& = \frac{1}{(\ell+d-1)(\ell-1)!} \inp{\boldsymbol z}{\boldsymbol w}^{\ell-1}\left( \frac{(\ell+d-1)}{\ell}  \inp{\boldsymbol z}{\boldsymbol w}\inp{\boldsymbol \zeta}{\boldsymbol e_j}- \inp{\boldsymbol z}{\boldsymbol e_j} \inp {\boldsymbol \zeta}{\boldsymbol w}\right)
      \end{align*}
      A direct verification shows that \begin{align*}
      \sum_{j=1}^d \partial_j \inp{\widehat{K}_\ell(\boldsymbol z,\boldsymbol w)\boldsymbol \zeta}{\boldsymbol e_j}  &=0,
   \end{align*} therefore, it follows that $\widehat{K}_\ell(\cdot,\boldsymbol w)\boldsymbol \zeta\in \widehat{\mathcal V}_\ell.$ Also, 
   \begin{align*}
    \inp{f}{\inp{\boldsymbol z}{\boldsymbol w}^{\ell-1}\boldsymbol z\overline{\boldsymbol w}^\dagger \boldsymbol e_i}_{\mathcal F_{\ell}}&=\sum_{j=1}^d \inp{f_j}{\inp{\boldsymbol z}{\boldsymbol w}^{\ell-1}\overline{w}_iz_j}_{\mathcal F}\\& =\sum_{j=1}^d w_i\inp{f_j}{\inp{\boldsymbol z}{\boldsymbol w}^{\ell-1}z_j}_{\mathcal F}\\&
    =\sum_{j=1}^d w_i\inp{\partial_jf_j}{\inp{\boldsymbol z}{\boldsymbol w}^{\ell-1}}_{\mathcal F}\\&
    =(\ell-1)!w_i\sum_{j=1}^d (\partial_jf_j)(\boldsymbol w)=0.
    \end{align*}
Thus, $\inp{f}{\widehat{K}_\ell(\cdot,\boldsymbol w)\boldsymbol e_i}_{\mathcal F_{\ell}}=
       \inp{f(\boldsymbol w)}{\boldsymbol e_i}$.
\end{proof}
The proposition below matching with Proposition \ref{qK1} is obtained by replacing $c(u) = \bar{u}$ by $c(u) = u$ is proved as before. 
\begin{prop} \label{sk}
Suppose that $K:\mathbb B_d \times \mathbb B_d \to \mathcal M_d(\mathbb C)$ is a sesqui-analytic 
Hermitian function satisfying the transformation rule with the multiplier  $c(u) = u$: 
\begin{equation}\label{eqnquasiinv**}
u K(u^{-1}\cdot \boldsymbol z, u^{-1}\cdot \boldsymbol w) \overline{u}^\dagger = K(\boldsymbol z, \boldsymbol w). 
\tag{$**$}\end{equation} 
Then  $K$ must be of the form 
\begin{align*}
K(\boldsymbol z, \boldsymbol z) &= \overline{u_{\boldsymbol z}}^\dagger \left (\begin{smallmatrix} D_1(r,r) & 0\\ 0& D_2(r,r) I_{d-1} \end{smallmatrix}  \right ) {u_{\boldsymbol z}},\, u_{\boldsymbol z} \in \mathcal U_{\boldsymbol z}, 
\end{align*}
where $D_i(r,r)$, $i=1,2$ are real analytic function on $[0,1)$ of the form $\sum_{m=0}^\infty \tilde{a}_{m,i}r^{2m}$ with $\tilde a_{0,1}=\tilde a_{0,2}$. 
\end{prop}
We need the following lemma similar to Lemma \ref{lem 4.10} to prove the main theorem describing sesqui-analytic Hermitian function quasi-invariant with multiplier $c(u) = u$. 
\begin{lemma} \label{lem 3.9}For any $u_{\boldsymbol z}\in \mathcal U_{\boldsymbol z}$,   we have 
$$ \overline{u_{\boldsymbol z}}^\dagger \left (\begin{matrix} D_1(r,r) & 0\\ 0& D_2(r,r) I_{d-1} \end{matrix}  \right ) {u_{\boldsymbol z}}=\big(D_1(r,r) - D_2(r,r)\big)\frac{{\boldsymbol z}  \overline{\boldsymbol z}^\dagger}{r^2} + D_2(r,r) I_d.$$
\end{lemma}
\begin{proof}
    The proof is similar to the proof of Lemma \ref{lem 4.10} except that we have to use the equality:  
\[\overline{u_{\boldsymbol z}}^\dagger E_{11}u_{\boldsymbol z}= \frac{\boldsymbol z \overline{\boldsymbol z}^\dagger}{r^2}.\qedhere\]
\end{proof}
\begin{theorem}\label{qK}
Any sesqui-analytic Hermitian function quasi-invariant with multiplier $c(u) = u$ is of the form
\[\widehat{K}^{(\alpha, \beta)}(\boldsymbol z, \boldsymbol w)=\sum_{j=0}^\infty \alpha_j \widehat{K}_{j}(\boldsymbol z,\boldsymbol w)+\sum_{j=1}^\infty \beta_j \widehat{K}_{j}^{\perp}(\boldsymbol z, \boldsymbol w),\,\, \alpha_j, \beta_j\in \mathbb C.\]
\end{theorem}
\begin{proof}
    As before, since the  kernel $\widehat{K}$ is sesqui-analytic Hermitian function quasi-invariant with multiplier $c(u)={u}$, it  must be of the form prescribed in Proposition \ref{sk}. Now, appealing to Lemma \ref{lem 3.9}, we obtain 
    \begin{equation} \label{kernelform2}
    \widehat{K}(\boldsymbol z, \boldsymbol w) =  \sum_{\ell=1}^{\infty} (\tilde{a}_{\ell,1} -\tilde{a}_{\ell,2} ) \inp{\boldsymbol z}{\boldsymbol w}^{\ell-1}\boldsymbol z \, \overline{\boldsymbol w}^\dagger + \sum_{\ell=0}^{\infty} \tilde{a}_{\ell,2}  \inp{\boldsymbol z}{\boldsymbol w}^\ell I_d \tag{$\sharp \sharp$}.\end{equation}
    The remaining portion of the proof is similar to that of Theorem \ref{thm positive definiteness}, where
    $\alpha_j=\tilde{a}_{j,2}j!$ and $\beta_j=(\tilde{a}_{j,1} (j+d-1)-\tilde{a}_{j,2} (d-1))(j-1)!$ for all $j.$
   \end{proof}
   
To determine among the kernels described in Theorem \ref{thm positive definiteness} (respectively, Theorem \ref{qK}), the ones that are non-negative definite, we recall a slight generalization of the criterion for non-negative definiteness of Farut-Koranyi \cite[Lemma 5.4]{Wallachset}:

\begin{lemma}[Lemma 5.1, \cite{BagchiHazraMisra}]\label{FK}
Let $\Omega$ be a domain in $\mathbb C^d$. Let $K:\Omega\times\Omega\to \mathcal M_n(\mathbb C)$ be a non-negative definite kernel and $\mathcal H_K(\Omega, \mathbb C^n)$ be the reproducing kernel Hilbert space determined by $K$. Suppose that $\mathcal H_K(\Omega, \mathbb C^n)$ can be decomposed as an orthogonal direct sum $\oplus_{\ell=0}^\infty \mathcal H_\ell$ and $K_\ell$ is the reproducing kernel of $\mathcal H_\ell$. Further assume that $\{c_\ell\}_{\ell \in \mathbb Z_+}$ is any sequence of complex numbers such that 
the sum $\sum_{\ell=0}^\infty c_\ell K_\ell(z,w)$ converges on $\Omega\times \Omega$. Then $\sum_{\ell=0}^\infty c_\ell K_\ell(z,w)$ is non-negative definite if and only if $c_\ell\geq 0$ for all $\ell \in \mathbb Z_+.$
\end{lemma}
Combining Faraut-Koranyi lemma with Theorem \ref{qK1} and Theorem \ref{qK}, we obtain a condition for a sesqui-analytic Hermitian function to be non-negative. 
\begin{theorem}\label{nonnegativity wrt multiplier u}
Suppose that $\tilde{K}^{(\alpha, \beta)}:\mathbb B_d \times \mathbb B_d \to \mathcal M_d(\mathbb C)$  is a sesqui-analytic Hermitian function  as 
in Theorem \ref{qK1} (respectively, $\widehat{K}^{(\alpha, \beta)}$ as in Theorem \ref{qK}). 
Then the kernel $\tilde{K}^{(\alpha, \beta)}$ (respectively, $\widehat{K}^{(\alpha, \beta)}$) is non-negative definite if and only if $\alpha_j \geq 0, \beta_j \geq 0$. 
\end{theorem}
\begin{proof}
In the expansion of $\tilde{K}^{(\alpha, \beta)}$ obtained in Theorem \ref{thm positive definiteness}, the kernels $\tilde{K}_j$ and $\tilde{K}_j^\perp$  are non-negative definite. Therefore, by Lemma \ref{FK}, we conclude that $\tilde{K}$
 is non-negative definite if and only if $\alpha_j \geq 0, \beta_j \geq 0$.
 The proof for $\widehat{K}^{(\alpha, \beta)}$ is similar and therefore omitted. 
\end{proof}
As a corollary of Theorem \ref{thm positive definiteness} (respectively, Theorem \ref{qK}), we prove that the restriction of the representation $\tilde{\pi}_\ell$  to $\tilde{\mathcal V}_\ell$ (respectively, restriction of $\hat{\pi}_\ell$  to $\hat{\mathcal V}_\ell$) is irreducible.
\begin{cor}\label{irr2} \begin{enumerate}
\item The restriction ${\tilde{\pi}_\ell}|_{{\tilde{\mathcal V}_\ell}}$ of  $\tilde{\pi}_\ell$ to the linear subspace $\tilde{\mathcal V}_\ell$ equipped with the restriction of the inner product $\inp{\cdot}{\cdot}_{\mathcal F_{\ell}}$ from $\mathbb C^d \otimes \mathcal P_\ell$ is irreducible.

\item The restriction ${\hat{\pi}_\ell}|_{{\widehat{\mathcal V}_\ell}}$ of  $\hat{\pi}_\ell$ to the linear subspace $\widehat{\mathcal V}_\ell$ equipped with the restriction of the inner product $\inp{\cdot}{\cdot}_{\mathcal F_{\ell}}$ from $\mathbb C^d \otimes \mathcal P_\ell$ is irreducible.
\end{enumerate}
\end{cor}
\begin{proof}
To prove part (1) of the corollary, suppose that there is a decomposition $\tilde{\mathcal V}_\ell = \mathcal V_\ell^1 \oplus \mathcal V_\ell^2$, where $\mathcal V_\ell^1$ and $\mathcal V_\ell^2$ are reducing subspaces for $\tilde{\pi}_\ell$. Let $K_\ell^1$ and $K_\ell^2$ be the kernel functions of $\mathcal V_\ell^1$ and $\mathcal V_\ell^2$, respectively.  Evidently, both $K_\ell^1$ and $K_\ell^2$ are quasi-invariant with respect to the same multiplier $\bar{u}$. It follows that $\tilde{K}_\ell = K_\ell^1 \oplus  K_\ell^2$. If $\ell =0$, then $\tilde{\mathcal V}_0= \{0\}$ and there is nothing to prove. Fix $\ell\in \mathbb N$, it follows from Theorem \ref{thm positive definiteness} that $K_\ell^1$ must be of the form $\sum_j \alpha_j \tilde{K}_j + \beta_j \tilde{K}_j^\perp$ for some choice of a set of non-negative numbers $\{\alpha_j\}$ and $\{\beta_j\}$. The Hilbert space determined by $\alpha_j \tilde{K}_j + \beta_j \tilde{K}_j^\perp$ contains the Hilbert space determined by $\alpha_j \tilde{K}_j$ as well as the one determined by $\beta_j \tilde{K}^\perp_j$. Now, if there is a non-zero $\alpha_j$ with $j\not = \ell$, then $\tilde{\mathcal V}_j$ must be a subspace of ${\mathcal V}^1_\ell$. Therefore $\alpha_j = 0$ except for $j=\ell$. A similar argument shows that $\beta_j = 0$ for all $j$. In consequence, if $\alpha_\ell>0$, then $\mathcal V_\ell^1=\tilde{\mathcal V}_\ell$, otherwise $\mathcal V_\ell^1=\{0\}$. 

The proof of part (2) of the Corollary is obtained exactly as in the proof of part (1) using Theorem \ref{qK}. 
\end{proof}

\subsection{Examples} \label{ex:3.3}
The examples discussed below shows that there are many quasi-invariant  kernels $K$ on $\mathbb B_d$ with multiplier of the form $c(u) = \bar{u}$ (resp. $c(u) = u$).  In these examples,  
the  monomials $\{\boldsymbol z^\alpha \otimes \boldsymbol \zeta:\alpha \in \mathbb Z^d_+, \boldsymbol \zeta \in \mathbb C^d\}$ are no longer orthogonal.  


Let $d\geq 2$. Recall that the Bergman kernel $B$ of the unit ball $\mathbb B$ is given by $B(\boldsymbol z, \boldsymbol w)=\frac{1}{(1-\inp{\boldsymbol z}{\boldsymbol w})^{d+1}}$. For $t\in \mathbb R$, we set 
$$B^{(t)}(\boldsymbol z, \boldsymbol w)=B^{t}\Big(\!\!\Big(\frac{\partial^2}{\partial {\boldsymbol z}_i\partial{\overline {\boldsymbol  w}}_j}\log B \Big)\!\!\Big)_{i,j=1}^d(\boldsymbol z, \boldsymbol w).$$
Clearly $B^{(t)}$ is a sesqui-analytic hermitian function for any real number $t$. It follows from \cite[Lemma 6.1]{GM} that  $B^{(t)}$ is quasi-invariant with the multiplier $c(u)=\overline{u}$. A direct computation shows that
 
\begin{align}\label{gWSeq2}
B^{(t)}(\boldsymbol z, \boldsymbol w)=
\frac{d+1}{(1-\left\langle \boldsymbol z, \boldsymbol w\right \rangle)^{t (d+1)+2}} \begin{pmatrix}
1-\sum_{j\neq 1}z_j\bar{w}_j&z_2\bar{w}_1&\cdots& z_d\bar{w}_1\\
z_1\bar{w}_2&1-\sum_{j\neq 2}z_j\bar{w}_j&\cdots& z_d\bar{w}_2\\
\vdots & \vdots& \vdots & \vdots\\
z_1\bar{w}_d &z_2\bar{w}_d & \cdots & 1-\sum_{j\neq d}z_j\bar{w}_j
\end{pmatrix}.
\end{align}
Thus \begin{equation}
    B^{(t)}(re_1 , re_1)=\frac{d+1}{(1-r^2)^{t (d+1)+2}} \begin{pmatrix}
1&0\\
0&(1-r^2)I_{d-1}
\end{pmatrix}, 0\leq r<1.
\end{equation}
Note that $B^{(t)}(0,0)=(d+1) I_d.$ Thus by Proposition \ref{qK1} we have  
  $B^{(t)}(\boldsymbol z, \boldsymbol z)= u_{\boldsymbol z}^\dagger B^{(t)}(re_1 , re_1) \overline{u_{\boldsymbol z}}$, where $r=\| \boldsymbol z\|$ and $u_{\boldsymbol z}$ is a unitary of the form $u_{\boldsymbol z}^* =\left ( \begin{smallmatrix} \tfrac{\boldsymbol z}{r} | & \bigstar  \end{smallmatrix} \right )$. Equivalently, \begin{equation}\label{transpose kernel}
    B^{(t)}(\boldsymbol z, \boldsymbol w)=\sum_{\ell=1}^\infty\big(a_{\ell,1} - a_{\ell,2}\big)\inp{\boldsymbol z}{\boldsymbol w}^{\ell-1}\overline{\boldsymbol w}  \boldsymbol z^\dagger +\sum_{\ell=0}^\infty a_{\ell,2}\inp{\boldsymbol z}{\boldsymbol w}^{\ell} I_d,
\end{equation}
where $a_{\ell,1}=(d+1) \frac{(t(d+1)+2)_{\ell}}{\ell!}$ and $ a_{\ell,2}=(d+1) \frac{(t(d+1)+1)_{\ell}}{\ell!}$ for all $\ell\in \mathbb Z_+$.
In this case it is easy to verify that $a_{\ell,1}\leq (\ell+1)a_{\ell,2}$ if and only if $t\geq 0.$ Therefore by Theorem \ref{nonnegativity wrt multiplier u} it follows that $B^{(t)}$ is a non-negative definite kernel if and only if $t\geq 0$. 

Since $B^{(t)}$ is quasi-invariant with respect to the multiplier $c(u)=\overline{u}$, it is easy to see that ${B^{(t)}}^\dagger$ is 
quasi-invariant with respect to the multiplier $c(u)=u$. Further, using \eqref{transpose kernel} and the identity $\frac{\inp{\boldsymbol z}{\boldsymbol w}^{\ell}}{\ell !}I_d=K_{\ell}+K_{\ell}^\perp$, we obtain  
\begin{equation}
    {B^{(t)}}^\dagger(\boldsymbol z, \boldsymbol w)=\sum_{\ell=1}^\infty\big((a_{\ell,1} - a_{\ell,2})(\ell+d-1)(\ell-1)!+a_{\ell,2}~\ell !\big)K_{\ell}^\perp(\boldsymbol z, \boldsymbol w) +\sum_{\ell=0}^\infty a_{\ell,2}~\ell !K_{\ell}(\boldsymbol z, \boldsymbol w).
\end{equation}
Hence it follows from Theorem \ref{nonnegativity wrt multiplier u} that the transpose ${B^{(t)}}^\dagger$ of the kernel ${B^{(t)}}$ is a non-negative definite kernel if and only if $t(d+1)+1\geq 0.$

Since $B^{(t)}$, $t \geq 0$, as well as ${B^{(t)}}^\dagger$, $t(d+1) + 1 \geq 0$, are non-negative definite, it follows from Proposition \ref{qi} that these kernels  are quasi-invariant but not invariant.

\section{$\mathcal U(d)$ Homogeneous operators}
\subsection{Boundedness and Irreducibility}
In this subsection, we derive explicit criterion for $\mathcal U(d)$-homogeneous $d$-tuple of multiplication operator $\boldsymbol M$ to be (a) bounded and (b) irreducible. This is done separately for the class of kernels of the form appearing in Theorem \ref{thm positive definiteness} and Theorem \ref{qK}.  
\begin{theorem} \label{thm-bounded}
Suppose that
$K:\mathbb B_d\times \mathbb B_d\to \mathcal M_d(\mathbb C)$ is a non-negative definite kernel of the form \eqref{kernelform1}. 
Then the $d$-tuple $\boldsymbol M$ on the Hilbert space $\mathcal H_K(\mathbb B_d, \mathbb C^d)$ is bounded if and only if $$\sup_{\ell} \left \{\frac{(\ell+1)a_{\ell-1,2}-a_{\ell-1,1}}{(\ell+1)a_{\ell,2}-a_{\ell,1}},\frac{a_{\ell-1,1}}{a_{\ell,1}} \right \}< \infty.$$ 
\end{theorem}
\begin{proof}
 The multiplication $d$-tuple $\boldsymbol M$ on the Hilbert space $\mathcal H_K(\mathbb B_d, \mathbb C^d)$ is bounded if and only if there exists $c>0$ such that $\big(c^2-\inp{\boldsymbol z}{\boldsymbol w}\big)K(\boldsymbol z, \boldsymbol w)$ is non-negative definite \cite[Lemma 2.7(ii)]{GM}.
 \begin{align*}
     \big(c^2-\inp{\boldsymbol z}{\boldsymbol w}\big)K(\boldsymbol z, \boldsymbol w)|_{\text{res~}\mathbb C^d \otimes \mathcal P_\ell}=&
     \big\{c^2\big(a_{\ell,1} - a_{\ell,2}\big)-\big(a_{\ell-1,1} - a_{\ell-1,2}\big)\big\}\inp{\boldsymbol z}{\boldsymbol w}^{l-1}\overline{\boldsymbol w} \boldsymbol z^\dagger \\&\phantom{AAAAAA} + \big(c^2 a_{\ell,2}-a_{\ell-1,2}\big)\inp{\boldsymbol z}{\boldsymbol w}^{l} I_d\\
     =& \big\{c^2\big((\ell+1)a_{\ell,2}-a_{\ell,1}\big)-\big((\ell+1)a_{\ell-1,2}-a_{\ell-1,1}\big)\big\}(\ell-1)!K_\ell\\
     &\phantom{AAAAAA} +\big(c^2a_{\ell,1}-a_{\ell-1,1}\big) \ell! K_\ell^\perp. 
 \end{align*}
 Hence by Lemma \ref{FK} $\big(c^2-\inp{\boldsymbol z}{\boldsymbol w}\big)K(\boldsymbol z, \boldsymbol w)$
is non-negative definite if and only if for all $l\in \mathbb N,$
\[c^2\big((\ell+1)a_{\ell,2}-a_{\ell,1}\big)-\big((\ell+1)a_{\ell-1,2}-a_{\ell-1,1}\big)\geq0\]and
\[c^2a_{\ell,1}-a_{\ell-1,1}\geq 0.\]
The claim of the theorem is clearly equivalent to these two positivity conditions completing the proof.   
\end{proof}
\begin{theorem} \label{thm-bounded1}
Suppose that
$K:\mathbb B_d\times \mathbb B_d\to \mathcal M_d(\mathbb C)$ is a non-negative definite kernel function of the form \eqref{kernelform2}.
Then the $d$-tuple $\boldsymbol M$ on the Hilbert space $\mathcal H_K(\mathbb B_d, \mathbb C^d)$ is bounded if and only if 

 $$\sup_{\ell} \left \{\frac{(\ell+d-1)\tilde{a}_{\ell-1,1}-(d-1)\tilde{a}_{\ell-1,2}}{(\ell+d-1)\tilde{a}_{\ell,1}-(d-1)\tilde{a}_{\ell,2}},\frac{\tilde{a}_{\ell-1,2}}{\tilde{a}_{\ell,2}} \right \}< \infty.$$ 
\end{theorem}

\begin{cor} 
Let $K$ be a non-negative definite kernel function either of the form \eqref{kernelform1} or \eqref{kernelform2}.  Assume that the $d$-tuple $\boldsymbol M$ on the Hilbert space $\mathcal H_K(\mathbb B_d, \mathbb C^d)$ is bounded. Then it is $\mathcal U(d)$-homogeneous.
\end{cor}
\begin{proof}
Since $K$ is quasi-invariant under $\mathcal U(d)$, the conlusion follows from Lemma \ref{multiplier}.
\end{proof}
\begin{theorem}
Let $d\geq 2$. Let $K$ be a non-negative definite  kernel function either of the form \eqref{kernelform1} or \eqref{kernelform2}. 
Assume that the $d$-tuple $\boldsymbol M$ on the Hilbert space $\mathcal H_K(\mathbb B_d, \mathbb C^d)$ is bounded. Then $\boldsymbol M$ is reducible if and only if $a_{\ell,1}=a_{\ell,2}$ or $\tilde{a}_{\ell,1}=\tilde{a}_{\ell,2}$ according as $K$ is of the form \eqref{kernelform1} or of the form \eqref{kernelform2}, $\ell \in \mathbb N$.
\end{theorem}
\begin{proof} First, let us consider the case of a kernel of the form \eqref{kernelform1}. Assume that $a_{\ell,1}=a_{\ell,2}$, $\ell\in \mathbb N$. Then $K(z,w)=\sum_{\ell=0}^\infty a_{\ell,2}\inp{\boldsymbol z}{\boldsymbol w}^\ell I_d$. Since $d\geq 2$, it is evident that the multiplication $d$-tuple $\boldsymbol M$ on $\mathcal H_K(\mathbb B_d, \mathbb C^d)$ is reducible. Conversely, assume that $\boldsymbol M$ on $\mathcal H_K(\mathbb B_d, \mathbb C^d)$ is reducible. 
 Since $K(\boldsymbol z,0)$ is constant and $\boldsymbol M$ is bounded, the discussion following Lemma 5.1 of \cite{GKoranyi}, there exists a non-trivial projection on $P$ on $\mathbb C^d$ such that $P K(\boldsymbol z,\boldsymbol w)=K(\boldsymbol z,\boldsymbol w) P.$ In case, $K$ is of the form \eqref{kernelform1}, this is equivalent to
\begin{equation} \label{red1}
    P\big(\sum_{\ell=1}^\infty\big(a_{\ell,1} - a_{\ell,2}\big)\inp{\boldsymbol z}{\boldsymbol w}^{\ell-1}\overline{\boldsymbol w} \boldsymbol z^\dagger\big)=\big(\sum_{\ell=1}^\infty\big(a_{\ell,1} - a_{\ell,2}\big)\inp{\boldsymbol z}{\boldsymbol w}^{\ell-1}\overline{\boldsymbol w}  \boldsymbol z^\dagger\big)P.
\end{equation}
Rewriting Equation \eqref{red1}, we have 
\begin{align*}
   0 = \sum_{\ell=1}^\infty\big(a_{\ell,1} - a_{\ell,2}) & \inp{\boldsymbol z}{\boldsymbol w}^{\ell-1} \big(P\overline{\boldsymbol w}  \boldsymbol z^\dagger- \overline{\boldsymbol w}  \boldsymbol z^\dagger P\big)\\
   &= \sum_{\ell=1}^\infty\big(a_{\ell,1} - a_{\ell,2}\big)\sum_{|\alpha|=\ell-1} \frac{|\alpha|!}{\alpha !} \sum_{i,j=1}^d (PE_{i,j}-E_{i,j}P) \boldsymbol z^{\alpha+\varepsilon_j} \bar{\boldsymbol w}^{\alpha +\varepsilon_i}.
    \end{align*}
Let $\ell \geq 1$ be fixed and choose $\alpha = (\ell -1) \varepsilon_i$, $1\leq i \leq d$. Then $\alpha + \varepsilon_j$ and $\alpha + \varepsilon_i$ are of the form   
$$(\ell-1) \varepsilon_i + \varepsilon_j,\,\, \ell\varepsilon_i, \, 1 \leq j \leq d,$$
respectively. If we choose any other multi-index $\beta\not = \alpha$ with $|\beta| = \ell-1$ and a pair of natural numbers $m,n$, $1\leq m,n \leq d$, then we can't have $\beta+ \varepsilon_m = \ell\varepsilon_i$ and $\beta + \varepsilon_n = (\ell-1) \varepsilon_i + \varepsilon_j $. It follows that the coefficients of 
 $z_i^{\ell-1} z_j \bar{w}_i^{\ell}$ must be zero. This means that $P$ must commute with all the elementary matrices $E_{i,j}$, $1\leq i,j \leq d$. Hence $P$ can not be a non-trivial projection contrary to our hypothesis unless $a_{\ell,1} = a_{\ell,2}$.  

If $K$ is of the form \eqref{kernelform2}, we have 
\begin{equation}\label{red2}
    P\big(\sum_{\ell=1}^\infty\big(\tilde{a}_{\ell,1} - \tilde{a}_{\ell,2}\big)\inp{\boldsymbol z}{\boldsymbol w}^{\ell-1}\boldsymbol z \overline{\boldsymbol w}^\dagger\big)=\big(\sum_{\ell=1}^\infty\big(\tilde{a}_{\ell,1} - \tilde{a}_{\ell,2}\big)\inp{\boldsymbol z}{\boldsymbol w}^{\ell-1} \boldsymbol z \overline{\boldsymbol w}^\dagger\big)P.
\end{equation}
Again, rewriting Equation \eqref{red2}, we have 
\begin{align*}
   0 = \sum_{\ell=1}^\infty\big(\tilde{a}_{\ell,1} - \tilde{a}_{\ell,2}\big) & \inp{\boldsymbol z}{\boldsymbol w}^{\ell-1} \big(P \boldsymbol z \overline{\boldsymbol w}^\dagger - \boldsymbol z \overline{\boldsymbol w}^\dagger  P\big)\\
   &= \sum_{\ell=1}^\infty\big(\tilde{a}_{\ell,1} - \tilde{a}_{\ell,2}\big)\sum_{|\alpha|=\ell-1} \frac{|\alpha|!}{\alpha !} \sum_{i,j=1}^d (PE_{i,j}-E_{i,j}P) \boldsymbol z^{\alpha+\varepsilon_i} \bar{\boldsymbol w}^{\alpha +\varepsilon_j}.
    \end{align*}
    Choosing $\alpha = (\ell-1)\varepsilon_i$, as before, we see that 
    $P$ can not be a non-trivial projection contrary to our hypothesis unless $\tilde{a}_{\ell,1} = \tilde{a}_{\ell,2}$. This completes the proof.
\end{proof}

\subsection{Computation of matrix coefficients and unitary equivalence}
 We wish to determine when the $d$-tuple $\boldsymbol M$ on the reproducing kernel Hilbert space $\mathcal H_K(\mathbb B_d, \mathbb C^d)$, where $K$ is given by either \eqref{kernelform1} or \eqref{kernelform2}, are  unitarily equivalent. For this, we rewrite the kernel K in the form 
$K(\boldsymbol z, \boldsymbol w) = \sum_{\alpha, \beta} A_{\alpha, \beta} \boldsymbol z^\alpha \boldsymbol{\bar{w}}^\beta$, where $\alpha, \beta \in \mathbb Z_+^d$ and $A_{\alpha, \beta}$ are $d \times d$ complex matrices. Since the kernels $K$ given in  \eqref{kernelform1} and \eqref{kernelform2} are normalized, any two $d$-tuple $\boldsymbol M$ acting on $\mathcal H_K(\mathbb B_d, \mathbb C^d)$ and $\mathcal H_{K^\prime}(\mathbb B_d, \mathbb C^d)$ are unitarily equivalent if and only if for all $\alpha, \beta$, $A_{\alpha, \beta}$ is unitarily equivalent to $A^\prime_{\alpha, \beta}$ by a fixed unitary $U$. Here we have taken $K^\prime(\boldsymbol z, \boldsymbol w) =  \sum_{\alpha, \beta} A^\prime_{\alpha, \beta} \boldsymbol z^\alpha \bar{\boldsymbol w}^\beta$. Therefore, we proceed to find the matrix coefficients $A_{\alpha, \beta}$.

We will first consider a non-negative definite kernel of the form 
\eqref{kernelform1}, that is, 
\begin{align*} K(\boldsymbol z, \boldsymbol w)
&= \sum_{\ell=1}^\infty\big(a_{\ell,1} - a_{\ell,2}\big)\inp{\boldsymbol z}{\boldsymbol w}^{\ell-1}\overline{\boldsymbol w} \cdot \boldsymbol z^\dagger +\sum_{\ell=0}^\infty a_{\ell,2}\inp{\boldsymbol z}{\boldsymbol w}^{\ell} I_d\\
&= \sum_{\ell=0}^\infty  \sum_{|\alpha|=\ell} {\binom{\ell}{\alpha}} \Big ( P_0(\ell) + \sum_{i,j=1}^d P_{i,j}(\ell+1) \boldsymbol z_j \bar{\boldsymbol w}_i \Big )\boldsymbol z^\alpha \bar{\boldsymbol w}^{\alpha}\\
&= \sum_{\alpha \in \mathbb Z_+^d} 
{\binom{|\alpha|}{\alpha}} P_0(|\alpha|) \boldsymbol z^\alpha \bar{\boldsymbol w}^\alpha + \sum_{\alpha \in \mathbb Z_+^d} 
\sum_{i,j} {\binom{|\alpha|}{\alpha}} P_{i,j} (|\alpha|+1) \boldsymbol z^{\alpha +\varepsilon_j} \bar{\boldsymbol w}^{\alpha +\varepsilon_i},
\end{align*}
where $P_0(|\alpha|) =a_{|\alpha|, 2}I_d $ and $P_{i,j}(|\alpha|) =(a_{|\alpha|,1} - a_{|\alpha|,2})E_{ij}$. The only monomials that occur in the kernel $K$ are of the form $\boldsymbol z^\alpha \bar{\boldsymbol w}^\beta$ with $\alpha - \beta = \varepsilon_j - \varepsilon_i$. To find the coefficient of such a monomial, we consider two cases, namely, $i\not = j$ and $i = j$. If $i\not = j$, then the coefficient $A_{\alpha+\varepsilon_j, \alpha+\varepsilon_i}$ of the monomial $\boldsymbol z^{\alpha +\varepsilon_j} \bar{\boldsymbol w}^{\beta+\varepsilon_i}$ is 
\begin{equation} \label{eqn:inotj}
A_{\alpha+\varepsilon_j, \alpha+\varepsilon_i} =  \binom{|\alpha|}{\alpha} P_{i,j} (|\alpha|+1),\,i\not = j.
\end{equation}
On the other hand if $i=j$, we have 
\begin{equation}\label{eqn:i=j}
A_{\alpha, \alpha}=  \binom{|\alpha| }{\alpha}P_0(|\alpha|) + \sum_{i=1}^{d}  \binom{|\alpha| -1}{\alpha - \varepsilon_i}  P_{i,i}(|\alpha|).
\end{equation} 

Replacing $P_0(|\alpha|)$ by $\tilde{P}_0(|\alpha|):=\tilde{a}_{|\alpha|, 2}I_d $ and $P_{i,j}(|\alpha|)$ by  $\tilde{P}_{i,j}(|\alpha|):= (\tilde{a}_{|\alpha|,1} - \tilde{a}_{|\alpha|,2})E_{ij}^\dagger$, we get the matrix coefficients for the kernel $K$ of the form \eqref{kernelform2}.  

\begin{theorem} \label{thm 4.18}
 Let $K$ and $K^\prime$ be two non-negative definite kernel function either of the form \eqref{kernelform1} or of the form  \eqref{kernelform2}. 
Assume that the $d$-tuples $\boldsymbol M$ on the Hilbert space $\mathcal H_K(\mathbb B_d, \mathbb C^d)$ and $\mathcal H_{K^\prime}(\mathbb B_d, \mathbb C^d)$ are bounded. Then 
 these two $d$-tuples are unitarily equivalent if and only if the two kernels $K$ and $K^\prime$ are equal. 
\end{theorem}
\begin{proof}
Since the kernels  $K$ and $K^\prime$ are normalized at $0$, it follows that the $d$-tuples $\boldsymbol M$ on two of these spaces are unitarily equivalent if and only if the matrix coefficients in the expansion of these kernels, as above, are unitarily equivalent via a fixed unitary $U$ of size $d\times d$, see \cite[Lemma 4.8 (c)]{Curtosalinas}. 
To prove the theorem, we first consider two kernels $K$ and $K^\prime$ of the form \eqref{kernelform1}, that is,
$$ K(\boldsymbol z, \boldsymbol w)
= \sum_{\ell=1}^\infty\big(a_{\ell,1} - a_{\ell,2}\big)\inp{\boldsymbol z}{\boldsymbol w}^{\ell-1}\overline{\boldsymbol w}  \boldsymbol z^\dagger +\sum_{\ell=0}^\infty a_{\ell,2}\inp{\boldsymbol z}{\boldsymbol w}^{\ell} I_d$$ and 
 $$K^\prime(\boldsymbol z, \boldsymbol w) = \sum_{\ell=1}^\infty\big(a^\prime_{\ell,1} - a^\prime_{\ell,2}\big)\inp{\boldsymbol z}{\boldsymbol w}^{\ell-1}\overline{\boldsymbol w}  \boldsymbol z^\dagger +\sum_{\ell=0}^\infty a^\prime_{\ell,2}\inp{\boldsymbol z}{\boldsymbol w}^{\ell} I_d.$$
 Assume that the $d$-tuples $\boldsymbol M$ on the Hilbert spaces 
$\mathcal H_K(\mathbb B_d, \mathbb C^d)$ and $\mathcal H_{K^\prime}(\mathbb B_d, \mathbb C^d)$ are unitarily equivalent. 
For fixed $\ell \in \mathbb Z_+$, set $a_\ell:=a_{\ell, 1} - a_{\ell,2}$ and $a^\prime_\ell:=a^\prime_{\ell,1} - a^\prime_{\ell,2}$.
It follows from Equation \eqref{eqn:inotj} that $a_\ell\, U E_{i,j} = a^\prime_\ell E_{i,j} U$ for every $i\not = j, \,1 \leq i,j \leq d$. Therefore  we conclude that $a_\ell$ and $a^\prime_\ell$ are simultaneously $0$ or not. If $a_\ell$ and $a^\prime_\ell$ are both zero for all $\ell$, then the two kernels $K$ and $K^\prime$ are invariant kernels of the form 
$\sum_\ell a_{\ell,2} I_d \inp{\boldsymbol{z}}{\boldsymbol{w}}^{\ell}$ and $\sum_\ell a^\prime_{\ell,2} I_d \inp{\boldsymbol{z}}{\boldsymbol{w}}^{\ell}$ respectively. Hence the $d$-tuples  $\boldsymbol M$ acting on $K$ and $K^\prime$ are unitarily equivalent if and only if $a_{\ell,2} = a^\prime_{\ell,2}$, for all $\ell$. 

Assume that $a_{\ell,1}\not = a_{\ell,2}$ for some $\ell \in \mathbb N$.  Fix one such $\ell$ and evaluate Equation \eqref{eqn:inotj} for a fixed pair $i,j$ with $i\not = j$. We then see that every column of the  $d\times d$ matrix $a_\ell U E_{i,j}$ is zero except for the $j$th column. This non-zero column is $a_\ell$ times the the $i$th column of $U$. On the other hand, each row of $d\times d$ matrix  $a^\prime_\ell E_{i,j} U$ is zero except for the $i$th one, which is $a^\prime_\ell$ times the  $j$th row of $U$. Since neither $a_\ell$ nor $a^\prime_\ell$ is zero, it follows that $U_{k,i} = 0$, $1\leq k\not = i \leq d$, similarly,  $U_{j,p} = 0$, $1\leq p\not = j \leq d$.  Hence $U$ must be a diagonal matrix.  Moreover, we have that 
$a_\ell U_{i,i} = a^\prime_\ell U_{j,j}$ for $1\leq i \not = j \leq d$. 
We claim $a_\ell = a^\prime_\ell$. For the proof, start with 
$a^2_\ell U_{i,i} = a_\ell(a^\prime_\ell U_{j,j}) = {a^\prime_\ell}^2 U_{i,i}$
and conclude that $a_\ell = a^\prime_\ell$. Hence $U_{i,i} = U_{j,j}$ for $i\not = j$ and it follows that
$U_{1,1} = U_{2,2} = U_{3,3} = \cdots = U_{d,d}$. In consequence, $U$ must be a unimodular scalar times identity.

If the kernels $K$ and $K^\prime$ are of the form \eqref{kernelform2}, then the proof is similar and therefore omitted. 
\end{proof}
The theorem below answers the question of unitary equivalence between two $\mathcal U(d)$-homogeneous multiplication tuples acting on  $\mathcal H_{K^\sharp}(\mathbb B_d, \mathbb C^d)$ and $\mathcal H_{K^{\sharp\sharp}}(\mathbb B_d, \mathbb C^d)$.
\begin{theorem} Let $K^{\sharp}$ be a kernel of the form \eqref{kernelform1} and $K^{\sharp \sharp}$ be a kernel of the form \eqref{kernelform2}. Assume that the $d$-tuples $\boldsymbol M$ on the Hilbert space $\mathcal H_{K^\sharp}(\mathbb B_d, \mathbb C^d)$ and $\mathcal H_{K^{\sharp\sharp}}(\mathbb B_d, \mathbb C^d)$ are bounded. Then 
\begin{enumerate}
\item if $d > 2,$ these two $d$-tuples are unitarily equivalent if and only if $a_{\ell,1} = a_{\ell,2} = \tilde{a}_{\ell,1} = \tilde{a}_{\ell,2},$ $\ell \in \mathbb N$.

 \item if $d=2$, these two $d$-tuples are unitarily equivalent if and only if $a_{\ell,1} = \tilde{a}_{\ell,2}$ and $a_{\ell,2} = \tilde{a}_{\ell,1}$, $\ell \in \mathbb N$.

\end{enumerate}
\end{theorem}

\begin{proof}
The idea of the proof of part (1) is the same as that of the proof for Theorem \ref{thm 4.18}. As in that proof, expanding $K^\sharp$ and $K^{\sharp\sharp}$ and assume that there is a unitary $U$ intertwining all the coefficients described in \eqref{eqn:inotj} and \eqref{eqn:i=j} with the ones described in the comments following these two equations. Assume that $a_{m,1} \not = a_{m,2}$ (and therefore $\tilde{a}_{m,1} \not = \tilde{a}_{m,2}$) for some $m \in \mathbb N$. For every fixed but arbitrary pair $(i,j)$, we must have 
$$(a_{m,1}-a_{m,2})\Big (\sum_{k,\ell=1}^d U_{k,\ell} E_{k,\ell} \Big ) E_{i,j} = (\tilde{a}_{m,1}-\tilde{a}_{m,2})E_{i,j}^\dagger \Big (\sum_{k,\ell=1}^d U_{k,\ell} E_{k,\ell} \Big ).$$
Since $E_{k,\ell} E_{i,j} = \delta_{\ell,i}E_{k,j}$, it follows  that 
$\sum_{k,l} U_{k,\ell} E_{i,j} = \sum_{k} U_{k,i}E_{k,j}$. Similarly, $E_{i,j}^\dagger \sum_{k,l} U_{k,\ell} = \sum_{\ell}U_{i,l}E_{j,l}.$
Thus for $j\not = i$, we have that $U_{i,j} =\lambda U_{j,i}$, $|\lambda| = 1$. Now, assume that $d > 2$. Moreover, for a fixed $k\not = i$, we have $U_{k,\ell} = 0 = U_{j,\ell}$,  and for fixed $\ell \not = j$, we have $U_{j,\ell}= 0 = U_{k,\ell}.$ Therefore for $d > 2$, we arrive at a contradiction unless $a_{\ell,1} = a_{\ell,2}$ and $\tilde{a}_{\ell,1} \not = \tilde{a}_{\ell,2}$ for all $\ell \in \mathbb N$, or that there is no unitary intertwiner.  

The proof of part (2) involves verifying that the unitary $\Big ( \begin{matrix}0 & 1\\-1 & 0 \end{matrix} \Big )$ intertwines 
the two kernels whenever $a_{\ell,1} = \tilde{a}_{\ell,2}$ and $a_{\ell,2} = \tilde{a}_{\ell,1}$, $\ell \in \mathbb N$.
\end{proof}

\subsection{Quasi-invariant diagonal kernels are invariant}
While there might be a characterization of all the invariant kernels on an arbitrary bounded symmetric domain $\Omega$, unfortunately, we haven't been able to find one. Therefore, we have decided to include a description of all the $\mathcal U(d)$-invariant kernels for the special case of $\Omega=\mathbb B_d$, the only case that we are able to resolve.  We begin by describing the kernels invariant under the group $\mathcal U(d)$.
\begin{prop}
Let $K:\mathbb B_d \times \mathbb B_d \to \mathcal M_n(\mathbb C)$ be a non-negative definite kernel. 
 Suppose $K$ is invariant  under $\mathcal U(d)$. Then $K$ must be of the form $K(\boldsymbol z, \boldsymbol w) = \sum_{\ell=0}^\infty A_\ell \inp{\boldsymbol z}{\boldsymbol w}^\ell,$ for some sequence $\{A_\ell\}_{\ell \in \mathbb Z_+}$ of positive definite $n\times n$ matrices. 
\end{prop}
\begin{proof}
Let $K(\boldsymbol z, \boldsymbol w)=\sum_{\alpha,\beta\in \mathbb Z_+^d}A_{\alpha, \beta}{\boldsymbol z}^\alpha\overline{{\boldsymbol w}}^\beta$, $\boldsymbol z, \boldsymbol w\in \mathbb B_d$.
Suppose that $K$ is invariant  under $\mathcal U(d)$, that is, $K(u\cdot \boldsymbol z, u\cdot \boldsymbol w)=K(\boldsymbol z, \boldsymbol w)$, for all $\boldsymbol z, \boldsymbol w\in \mathbb B_d$ and $u \in \mathcal U(d)$. Choosing $u$ to be the diagonal unitary matrices $\mbox{diag}(e^{i\theta_1},\ldots,e^{i\theta_d})$, $\theta:=(\theta_1,\ldots,\theta_d)\in \mathbb R^d$, we get that
\begin{equation*}
\sum_{\alpha,\beta\in \mathbb Z_+^d}A_{\alpha, \beta}{\boldsymbol z}^\alpha\overline{{\boldsymbol w}}^\beta 
e^{i(\alpha-\beta)\cdot \theta}
=\sum_{\alpha,\beta\in \mathbb Z_+^d}A_{\alpha, \beta}{\boldsymbol z}^\alpha\overline{{\boldsymbol w}}^\beta, \boldsymbol z, \boldsymbol w\in \mathbb B_d,
\end{equation*}
where $(\alpha-\beta)\cdot \theta:=  (\alpha_1-\beta_1)\theta_1+\cdots+(\alpha_d-\beta_d)\theta_d$. Therefore we have 
\begin{equation}\label{diaguitary}
A_{\alpha,\beta}(e^{i((\alpha-\beta)\cdot \theta)}-1)=0, ~\mbox{for all~}\alpha,\beta\in \mathbb Z_+^d,~ \theta\in \mathbb R^d.
\end{equation}
Let $\alpha, \beta\in \mathbb Z_+^d$ and $\alpha\neq \beta.$ Then there exists $m$, $1\leq m \leq d$, such that $\alpha_m\neq \beta_m.$
Choosing $\theta_j=0$ for all $j\neq m$ in \eqref{diaguitary}, we obtain  that $A_{\alpha,\beta}=0$. 
Hence $K(\boldsymbol z, \boldsymbol w)$ is of the form $\sum_{\alpha \in \mathbb Z_+^d}A_{\alpha, \alpha}{\boldsymbol z}^\alpha\overline{{\boldsymbol w}}^\alpha.$ Now choosing $u$ to be $u_{\boldsymbol z}$, we see that
\begin{equation*}
    K(\boldsymbol z, \boldsymbol z)=K(u_{\boldsymbol z}\cdot\boldsymbol z, u_{\boldsymbol z}\cdot\boldsymbol z)=K (\|\boldsymbol z\|e_1,\|\boldsymbol z\|e_1)=\sum_{\ell=0}^\infty A_{\ell\epsilon_1,\ell\epsilon_1}\|\boldsymbol z\|^{2\ell}.
\end{equation*}
By polarization, we get that $K(\boldsymbol z, \boldsymbol w)=\sum_{\ell=0}^\infty A_{\ell\epsilon_1,\ell\epsilon_1}{\inp{\boldsymbol z}{\boldsymbol w}}^\ell =\sum_{\ell=0}^\infty \tilde{A}_{\ell}{\inp{\boldsymbol z}{\boldsymbol w}}^\ell$, where $ \tilde{A}_{\ell} =A_{\ell\epsilon_1,\ell\epsilon_1}.$ Since $K$ is non-negative definite, by \cite[Lemma 4.1 (c)]{Curtosalinas}, it follows that $\tilde{A}_{\ell}$ is positive definite, completing the proof.
\end{proof} 

For any $u$ in $\mathcal U(d)$ and $\alpha\in \mathbb Z^d_+$ with $|\alpha|=\ell$, let $X^u_{\alpha,\beta}$, $\beta \in \mathbb Z^d_+, |\beta|=\ell$, be the complex numbers given by
\begin{equation}
    (u \cdot \boldsymbol z)^{\alpha}=\sum_{|\beta|=\ell} X^u_{\alpha,\beta} \boldsymbol z^\beta.
\end{equation}
We arrive at the same conclusion as that of Proposition \ref{diaguitary} even if we assume that  $K$ is merely a quasi-invariant diagonal kernel. 
For the proof, we begin by proving a couple of preparatory lemmas. 
\begin{lemma}\label{lemma unitary rep}
For any $u\in \mathcal U(d),$  the matrix $\big(\!\big((\frac{\beta!}{\alpha!})^{\frac{1}{2}}X^u_{\alpha,\beta}\big)\!\big)_{|\alpha|=|\beta|=\ell}$ is unitary.
\end{lemma}
\begin{proof}
Consider the space of homogeneous polynomials $\mathcal P_\ell
$ endowed with the Fischer-Fock inner product. Note that $\{\frac{\boldsymbol z^{\gamma}}{(\gamma!)^{\frac{1}{2}}} \}_{|\gamma|=\ell}$ forms an orthonormal basis of $\mathcal P_\ell$ and $\big(\!\big((\frac{\beta!}{\alpha!})^{\frac{1}{2}}X^u_{\alpha,\beta}\big)\!\big)_{|\alpha|=|\beta|=\ell}$ is the matrix representation of the unitary map $p \to p\circ u$ with respect to this orthonormal basis.
\end{proof}

\begin{lemma}\label{special unitary}
There exists a unitary $u\in \mathcal U(d)$ such that $X^{u}_{\ell \varepsilon_1,\alpha}\neq 0$ for all $\alpha\in \mathbb Z^d_+$ with $|\alpha|=\ell.$
\end{lemma}
\begin{proof}
Choose a unitary $u=(u_{ij})_{i,j=1}^d$ in $\mathcal U(d)$ such that $u_{1j}\neq 0$ for $j=1,\ldots,d.$ Since 
$$(u\cdot \boldsymbol z)^{\ell \varepsilon_1}= (u_{11}z_1+\cdots+u_{1d}z_d)^{\ell}=\sum_{|\alpha|=\ell} \frac{\ell!}{\alpha !}u_{11}^{\alpha_1}\ldots u_{1d}^{\alpha_d}~\boldsymbol z^{\alpha},~\alpha=(\alpha_1,\ldots,\alpha_d)\in \mathbb Z^d_+,$$
we get that $X^u_{\ell\varepsilon_1,\alpha}=\frac{\ell!}{\alpha !}u_{11}^{\alpha_1}\ldots u_{1d}^{\alpha_d}$, which is certainly non-zero by our choice of $u$.
\end{proof}
We now prove the main theorem of this section stated below using Lemma \ref{lemma unitary rep} and Lemma \ref{special unitary}. 
\begin{theorem}\label{thm4.10}
Let $\mathcal H\subset {\rm Hol}(\mathbb B_d, \mathbb C^n)$ be a reproducing kernel Hilbert space. Suppose that $\mathbb C^n$-valued  polynomials are dense in $\mathcal H$ and $\inp{\boldsymbol z^{\alpha}\otimes \boldsymbol \xi}{\boldsymbol z^{\beta}\otimes \boldsymbol \eta}=0$, for all $\alpha \neq \beta$ in $\mathbb Z^d_+$ and $\boldsymbol \xi, \boldsymbol \eta$ in $\mathbb C^n$. If the $d$-tuple $\boldsymbol{M}$ on $\mathcal H$ is $\mathcal U(d)$-homogeneous, then there exists a sequence of positive definite $n\times n$ matrices $\{A_\ell\}_{\ell \in \mathbb Z_+}$ such that \[\|\boldsymbol z^{\alpha}\otimes \boldsymbol \xi\|^2= {\alpha!}\inp{A_{|\alpha|}\boldsymbol \xi}{\boldsymbol \xi}, \qquad \alpha \in \mathbb Z^d_+, \; \boldsymbol \xi \in \mathbb C^n.\]
\end{theorem}
\begin{proof}
Since $\boldsymbol{M}$ on $\mathcal H$ is $\mathcal U(d)$-homogeneous, by Lemma \ref{multiplier}, for each $u\in \mathcal U(d)$ there exists a unitary $\Gamma(u)$ on $\mathcal H$ of the form 
$$\Gamma(u)(f)=c(u)f\circ u,~f\in \mathcal H,$$
where $c(u)\in \mathcal U(n)$ for all $u\in \mathcal U(d)$.
Let $\ell\in \mathbb Z_+$. 
For $\alpha, \beta\in \mathbb Z_+^d$ with $|\alpha|=|\beta|=\ell$, $\alpha \neq \beta$, and $\boldsymbol \xi, \boldsymbol \eta\in \mathbb C^n$, we have
\begin{align}\label{Diagonal Thm}
    \inp{\Gamma (u)(\boldsymbol z^{\alpha}\otimes \boldsymbol \xi)}{\Gamma (u)(\boldsymbol z^{\beta}\otimes \boldsymbol \eta)}&=\inp{(u\cdot \boldsymbol z)^{\alpha}\otimes c(u)\boldsymbol \xi}{(u\cdot \boldsymbol z)^{\beta}\otimes c(u) \boldsymbol \eta}\nonumber\\
    &=\inp{\sum_{|\gamma|=\ell}X^u_{\alpha,\gamma}{\boldsymbol z}^{\gamma}\otimes c(u)\xi}{\sum_{|\delta|=\ell}X^u_{\beta,\delta}{\boldsymbol z}^{\delta}\otimes c(u)\boldsymbol \eta}\nonumber\\
    &=\sum_{|\gamma|=\ell}X^u_{\alpha,\gamma}\overline{X^u_{\beta,\gamma}}\inp{{\boldsymbol z}^{\gamma}\otimes c(u)\boldsymbol \xi}{{\boldsymbol z}^{\gamma}\otimes c(u)\boldsymbol \eta}.
\end{align}
Since $\Gamma(u)$ is unitary and $\inp{\boldsymbol z^{\alpha}\otimes \boldsymbol \xi}{\boldsymbol z^{\beta}\otimes \boldsymbol \eta}=0$, it follows  that $\inp{\Gamma (u)({\boldsymbol z}^{\alpha}\otimes \boldsymbol \xi)}{\Gamma (u)({\boldsymbol z}^{\beta}\otimes \boldsymbol \eta)}=0.$ Hence from \eqref{Diagonal Thm} we obtain
\begin{equation} 
    \sum_{|\gamma|=\ell}X^u_{\alpha,\gamma}\overline{X^u_{\beta,\gamma}}\inp{{\boldsymbol z}^{\gamma}\otimes c(u)\boldsymbol \xi}{{\boldsymbol z}^{\gamma}\otimes c(u)\boldsymbol \eta}=0.
\end{equation}
Since $c(u)$ is unitary and the above equality holds for all $\boldsymbol\xi, \boldsymbol\eta\in \mathbb C^n$, we get
\begin{equation}\label{eqn orthogonal1}
    \sum_{|\gamma|=\ell}X^u_{\alpha,\gamma}\overline{X^u_{\beta,\gamma}}\inp{{\boldsymbol z}^{\gamma}\otimes \boldsymbol \xi}{{\boldsymbol z}^{\gamma}\otimes \boldsymbol \eta}=0.
    \end{equation}
    By Lemma \ref{special unitary}, there exists a unitary $u_0\in \mathcal U(d)$ such that $X^{u_0}_{\ell\varepsilon_1,\gamma}\neq 0$ for all $\gamma$ with $|\gamma|=\ell.$ 
Choosing $\alpha=\ell\varepsilon_1$ and $u=u_0$ in  \eqref{eqn orthogonal1}, we get for all $\beta\neq \ell\varepsilon_1$ with $|\beta|=\ell$,
\begin{equation}\label{eqn orthogonal}
\sum_{|\gamma|=\ell}X^{u_0}_{\ell\varepsilon_1,\gamma}\langle {\boldsymbol z}^{\gamma}\otimes \boldsymbol \xi, {\boldsymbol z}^{\gamma}\otimes \boldsymbol \eta\rangle~ \overline{X^{u_0}_{\beta,\gamma}}=0.
\end{equation}
Hence it follows from Lemma \ref{lemma unitary rep} that 
$$ X^{u_0}_{\ell\varepsilon_1,\gamma}\langle {\boldsymbol z}^{\gamma}\otimes \boldsymbol \xi,{\boldsymbol z}^{\gamma}\otimes \boldsymbol\eta\rangle= \chi_{\ell,\boldsymbol\xi, \boldsymbol\eta} ~\gamma ! X^{u_0}_{\ell\varepsilon_1,\gamma},$$
that is,
$\langle {\boldsymbol z}^{\gamma}\otimes \boldsymbol \xi,{\boldsymbol z}^{\gamma}\otimes \boldsymbol\eta\rangle= \chi_{\ell,\boldsymbol\xi, \boldsymbol\eta}~ \gamma !,$
for all $\gamma$ with $|\gamma|=\ell$ and for some constant $\chi_{\ell,\boldsymbol \xi, \boldsymbol\eta}$. Clearly there exists a $n\times n$ positive definite matrix $A_\ell$ such that 
$$\langle A_\ell \boldsymbol \xi,\boldsymbol \eta\rangle_{\mathbb C^n}=\chi_{\ell,\boldsymbol\xi, \boldsymbol\eta},~\boldsymbol\xi, \boldsymbol\eta\in \mathbb C^n.$$ This completes the proof.
\end{proof}
 As a corollary, we conclude that a quasi-invariant non-negative definite diagonal kernel defined on the Euclidean ball must necessarily be invariant. 
\begin{cor} \label{diagonalkernel} Let $K:\mathbb B_d \times \mathbb B_d \to \mathcal M_n(\mathbb C)$ be a non-negative definite kernel such that $\partial ^{\alpha}\bar{\partial}^{\beta}K(0,0)=0$ whenever $\alpha \not = \beta.$  Suppose that $\mathbb C^n$-valued  polynomials are dense in $\mathcal H_K(\mathbb B_d,\mathbb C^n).$ If $K$ is quasi-invariant under $\mathcal U(d)$ then it must be of the form $K(\boldsymbol z,\boldsymbol w)= \sum_{\ell}A_\ell^{-1}\frac{\inp{\boldsymbol z}{\boldsymbol w}^{\ell}}{\ell!},$ where $A_\ell$  is a positive invertible $n\times n$ matrix for all $\ell\in \mathbb Z_+$.
\end{cor} 

\begin{proof}
Since $K$ is quasi-invariant under $\mathcal U(d)$,  by Lemma \ref{multiplier}, the $d$-tuple  $\boldsymbol{M}$ on $\mathcal H_K(\mathbb B_d,\mathbb C^n)$  is $\mathcal U(d)$-homogeneous.  It follows from Theorem \ref{thm4.10} that the set \[  \left \{\frac{1}{\sqrt{\alpha !}} \boldsymbol z^\alpha A_{\ell}^{-1/2} \varepsilon_i: 1\leq i \leq n, |\alpha| = \ell \right\}\] forms  an orthonormal basis for the space of $\mathbb C^n$-valued  homogeneous polynomial $\mathcal P_\ell \otimes \mathbb C^n$ in $\mathcal H_K(\mathbb B_d,\mathbb C^n),$
where  $A_\ell$  is a positive definite invertible $n\times n$ matrix for all $\ell\in \mathbb Z_+$.  Equivalently,  the reproducing kernel $K_{\ell}$ of  the (finite dimensional) Hilbert space $\mathcal P_\ell \otimes \mathbb C^n$ is given by the formula: 
$$K_\ell(\boldsymbol z, \boldsymbol w) = A_\ell^{-1} \frac{\inp{\boldsymbol z}{\boldsymbol w}^{\ell}}{\ell!}.$$  Thus the kernel $K$ must be of the  $K(\boldsymbol z,\boldsymbol w)= \sum_{\ell}A_\ell^{-1} \frac{\inp{\boldsymbol z}{\boldsymbol w}^{\ell}}{\ell!}$ for all $\boldsymbol z, \boldsymbol w\in \mathbb B^d.$ This proves the result.
\end{proof}
There are several separate equivalent assertions that are implicit in the previous corollary. We list them below. 
\begin{enumerate}
\item the  inner product on $\mathcal P_\ell\otimes \mathbb C^n$ is given by the usual Hilbert space tensor product of the two finite dimensional Hilbert spaces, namely, $\big(\mathcal P_\ell, \inp{\cdot}{\cdot}_{\mathcal F_{\ell}}\big)$ and $\big(\mathbb C^n, \inp{\cdot}{\cdot}_{A_\ell}\big)$, where $\inp{\boldsymbol \xi}{\boldsymbol \eta}_{A_\ell} = \inp{A_\ell \boldsymbol \xi}{\boldsymbol \eta}_{\mathbb C^n}$.
\item The set $\left \{\frac{1}{\sqrt{\alpha !}} \boldsymbol z^\alpha A_{\ell}^{-1/2} \varepsilon_i: 1\leq i \leq n, |\alpha| = \ell \right\}$ form an orthonormal basis for $\mathcal P_\ell \otimes \mathbb C^n.$
\item The kernel function $K_\ell$ on the (finite dimensional) Hilbert space $\big(\mathcal P_\ell, \inp{\cdot}{\cdot}_{\mathcal F_{\ell}}\big) \otimes \big(\mathbb C^n, \inp{\cdot}{\cdot}_{A_\ell}\big)$ is given by the formula: 
$$K_\ell(\boldsymbol z, \boldsymbol w) = A_\ell^{-1} \frac{\inp{\boldsymbol z}{\boldsymbol w}^{\ell}}{\ell!},\; \boldsymbol z, \boldsymbol w\in \mathbb B^d.$$
\item The kernel function $K$ of the Hilbert space $\mathcal H_K(\mathbb B_d,\mathbb C^n)$ is of the form $K(\boldsymbol z,\boldsymbol w)= \sum_{\ell}A_\ell^{-1} \frac{\inp{\boldsymbol z}{\boldsymbol w}^{\ell}}{\ell!}$.
\end{enumerate}
\section{Classification} 
Before we discuss the question of classification of $\mathcal U(d)$-homogeneous operators, we  note that some of our results exist in the representation theory literature albeit somewhat disguised. We believe unraveling this relationship would serve a useful purpose.   


\subsection{Decomposition of tensor product of $\pi_1\otimes \pi_\ell$ and $\bar{\pi}_1\otimes \pi_\ell$}  There is an alternative but equivalent description of the representations  $\tilde \pi_{\ell}$ and $\hat{\pi}_{\ell}$, given below, which is also useful. For this, we identify the space of linear polynomials $\mathcal P_1$ as the dual of the linear space $\mathbb C^d$.
We define $\phi: \mathbb C^d \otimes \mathcal P_\ell \to \mathcal P_1\otimes \mathcal P_\ell$ by setting 
$$\phi \Big (\sum_{i=1}^d \boldsymbol e_i p^i_\ell \Big )(\boldsymbol z, \boldsymbol w)= \sum_{i=1}^d z_i p^i_{\ell}(\boldsymbol w),\, \boldsymbol z,\, \boldsymbol w \in \mathbb B_d.$$ 

Therefore we see that $\mbox{\rm Im}\, (\phi)$ is the space $\mathcal P_1 \otimes \mathcal P_{\ell}$ of homogeneous polynomials of degree $\ell +1$ in $2d$-variables.  Since the monomials $z_1, \ldots , z_d$ form an orthonormal basis in $\mathcal P_1$ with respect to the Fisher-Fock inner product, it follows that $\phi$ is unitary. Hence, $ \tilde{\pi}_{\ell}$ is unitarily equivalent, via $\phi$, with  $\pi_1\otimes \pi_{\ell}$, where $$\left(\pi_1(u)\otimes \pi_{\ell}(u)\right)p(\boldsymbol z, \boldsymbol w)=p(u^{-1} \cdot \boldsymbol z, u^{-1} \cdot \boldsymbol w), \; p \in \mathcal P_1 \otimes \mathcal P_{\ell}.$$ 

The contragredient of  the representation $\pi_1$ is the defined to be the representation $\overline{\pi}_1(u)p_1(\boldsymbol z):=p_1(u^{\dagger }\cdot \boldsymbol z)$, $p_1 \in \mathcal P_1$, 
 we have 
$$ (\overline{\pi}_1(u)\otimes {\pi}_\ell(u)) p(\boldsymbol z, \boldsymbol w) = p( u^{\dagger} \boldsymbol \cdot \boldsymbol z , u^{-1} \cdot \boldsymbol w), p\in \mathcal P_1 \otimes \mathcal P_\ell.$$
Again, $\phi$ intertwines $\hat{\pi}_{\ell}$ and $\overline{\pi}_1\otimes {\pi}_\ell$:
\begin{align*}
(\phi \hat{\pi}_{\ell}(u))f(\boldsymbol w) &=
    \sum_{i=1}^d z_i (u(f\circ u^{-1}))_i(\boldsymbol w)\\&=\inp{u (f\circ u^{-1})(\boldsymbol w)}{\overline{\boldsymbol z}}_{\mathbb C^d}\\&=\inp{ (f\circ u^{-1})(\boldsymbol w)}{\overline{u^\dagger\cdot \boldsymbol z}}_{\mathbb C^d}
\\&=\sum_{i=1}^d (u^\dagger\cdot \boldsymbol z)_i f_i(u^{-1}\cdot \boldsymbol w)=(\overline{\pi}_1(u)\otimes {\pi}_\ell(u)) \phi(f)(\boldsymbol w),
\end{align*}
where $u\in \mathcal U(d)$ and $ f=\left(\begin{smallmatrix}f_1\\
\vdots\\
f_d
\end{smallmatrix}\right)\in  \mathbb C^d \otimes \mathcal P_\ell$.

Let $\overline{S}_\ell = (\mathcal P_\ell, \pi_\ell)$ and $S_1 = (\mathcal P_1, \overline{\pi}_1)$. Note that in the standard terminology of representation theory, the representation $\tilde{\pi}_\ell\sim_u\pi_1 \otimes \pi_\ell$ is $\bar{S}_1\otimes \bar{S}_\ell$, where $\sim_u$ stands for unitary equivalence of the two representations.  Similarly,  $\hat{\pi}_\ell\sim_u\bar{\pi}_1 \otimes \pi_\ell$ is $S_1\otimes \bar{S}_\ell$.
From Equation (23.12) of \cite{MTaylor}, we see that 
\begin{equation} \label{eq:5.1} 
S_1 \otimes \overline{S}_\ell = D_{(1,0,\ldots, 0,-\ell)} \oplus D_{(0, \ldots , 0, 1 - \ell)},
\end{equation}
where $D_{(0,\ldots, 0,1-\ell)} \sim_u \overline{S}_{\ell-1}$ and using Proposition 23.3 of \cite{MTaylor}, it follows that $D_{(1,0,\ldots, 0,-\ell)}$ is unitarily equivalent to the restriction of the representation 
$\hat{\pi}_\ell$ to the subspace $\widehat{\mathcal V_\ell} \subset \mathbb C^d \otimes \mathcal P_\ell$ via the map $\phi$.
Note that the restriction of $\hat{\pi}_\ell$ to $\widehat{\mathcal V}_\ell$ is irreducible (refer to Corollary \ref{irr2}) and the representation $\hat{\pi}_\ell$ has exactly two irreducible components, see \eqref{eq:5.1}. Therefore, we have proved the following theorem. 

\begin{theorem}\label{T23.3}
The subspaces $\widehat{\mathcal V}_\ell$ and $\widehat{\mathcal V}_\ell^\perp$ of $\mathbb C^d \otimes \mathcal P_\ell$ are reducing for the representation $\hat{\pi}_{\ell}$, moreover, the restriction of $\hat{\pi}_{\ell}$ to these subspaces are irreducible. 
\end{theorem}

One would like to obtain a similar decomposition of $\tilde{\pi}_\ell$ into irreducible  representations as in Theorem \ref{T23.3}. However, such a decomposition appears to be not available in any explicit form. This, we provide below. Clearly,  
$$\mbox{\rm Im}\, (\phi) = \phi(\tilde{\mathcal V}_\ell)\oplus \phi(\tilde{\mathcal V}_\ell^\perp),$$
where 
\begin{enumerate}
    \item $\phi(\tilde{\mathcal V}_\ell) = \{p(\boldsymbol z, \boldsymbol w) =\sum_{i=1}^d z_i p_\ell^i (\boldsymbol w) \in \mathcal P_1 \otimes  \mathcal P_{\ell}  : p_{|\mbox{\rm res}\, \Delta} = 0\}$, where $\Delta:=\{(\boldsymbol z, \boldsymbol z): \boldsymbol z\in \mathbb B_d\}$, 
    \item $\phi(\tilde{\mathcal V}_\ell^\perp)= \{ \sum_{i=1}^d z_i \partial_i q_{\ell+1} (\boldsymbol w) \in \mathcal P_1 \otimes  \mathcal P_{\ell} : q_{\ell+1} \in \mathcal P_{\ell+1}\}.$
\end{enumerate}
Also, we note that $\phi(\tilde{\mathcal V}_\ell^\perp) = \{p_{|\mbox{\rm res}\, \Delta}: p\in \mathcal P_1 \otimes  \mathcal P_{\ell}\}$.
Since $\tilde{\mathcal V}_\ell$ is invariant under $\tilde{\pi}_{\ell}$ and $\phi$ is an intertwining map between $\tilde{\pi}_{\ell}$ and $\pi_1\otimes \pi_{\ell}$, it follows that $\phi(\tilde{\mathcal V}_\ell)$ is invariant under $\pi_1\otimes \pi_{\ell}$. 
Let $R:\mathcal P_1  \otimes  \mathcal P_\ell \to \mathcal P_{\ell+1}$ be the restriction map, that is, $R p(\boldsymbol z, \boldsymbol w) := p(\boldsymbol z, \boldsymbol z) =\sum_{i=1}^d z_i p_{\ell}^i (\boldsymbol z).$ 
Thus we have proved the lemma that follows. 
\begin{lemma} \label{l+1}
The map $R$ on  $\phi(\tilde{\mathcal V}_\ell^\perp)$ is onto  $\mathcal P_{\ell+1}$ and is isometric when $\mathcal P_{\ell+1}$ is equipped with the Fischer-Fock inner product. Moreover, $R (\pi_1(u) \otimes \pi_{\ell}(u)) R^* = \pi_{\ell+1}(u)$.  
\end{lemma}
As before, since $\pi_{\ell+1}$ is an irreducible representation, the proof of the theorem stated below follows from Lemma \ref{l+1}.
\begin{theorem} \label{irr1}
The subspaces $\tilde{\mathcal V}_\ell$ and $\tilde{\mathcal V}_\ell^\perp$ of $\mathbb C^d \otimes \mathcal P_\ell$ are reducing for the representation $\tilde{\pi}_{\ell}$, moreover, the restriction of $\tilde{\pi}_{\ell}$ to these subspaces are irreducible. 
\end{theorem}

We point out that half of Theorems \ref{T23.3} and \ref{irr1} has been already proved in Corollary \ref{irr2}. The remaining half can also be proved in a similar manner to that of the proof in Corollary \ref{irr2}. 
However, we believe the proof we have given here is more revealing. 

Recall the decomposition of $\tilde{K}^{(\alpha, \beta)}(\boldsymbol z, \boldsymbol w)$ given in Theorem \ref{thm positive definiteness}. Let $\Lambda=\mathbb{Z}_+$. For $\lambda \in \Lambda,$ choosing  
\begin{align*}
    b_{\lambda}=&\begin{cases}\alpha_j ,\,\mbox{if}\,\,\lambda=2j+1\\
    \beta_j, \,\mbox{if}\,\,\lambda=2j, 
    \end{cases}
\end{align*}
and setting \begin{align*}
    K_{\lambda}(\boldsymbol z, \boldsymbol w)=&\begin{cases}\tilde{K}_j(\boldsymbol z, \boldsymbol w) ,\,\mbox{if}\,\,\lambda=2j+1\\
    \tilde{K}_j^\perp(\boldsymbol z, \boldsymbol w), \,\mbox{if}\,\,\lambda=2j,
    \end{cases}
\end{align*}
we obtain a second decomposition of the kernel $\tilde{K}^{(\alpha, \beta)}$ from Theorem \ref{3main} that coincides with the previous one from Theorem \ref{thm positive definiteness}. A similar statement can be made about the kernel $\widehat{K}^{(\alpha, \beta)}$ appearing in Theorem \ref{qK}.

\subsection{Classification} The natural action of the unitary group $\mathcal U(d)$ on $\mathbb C^d \otimes \mathcal P$ associated with the multiplier $c$ is given by $ p \to c(u)(p\circ u^{-1}), \; p \in \mathbb C^d \otimes \mathcal P ~\mbox{and}~ u \in \mathcal U(d).$ We obtain two classes of $\mathcal U(d)$-homogeneous $d$-tuple of operators with respect to two different multipliers $c(u)=\bar{u}$ (see Theorem \ref{thm positive definiteness}) and $c(u)=u$ (see Theorem \ref{qK}). The map $u \mapsto \bar{u}$ and $u \mapsto u$ are $d$-dimensional irreducible unitary representations of the group $\mathcal U(d).$

The classification of finite dimensional irreducible unitary representations of the unitary group $\mathcal U(n)$ is well studied. The result is summarized in \cite[Proposition 22.2]{MTaylor} and is reproduced below for  ready reference.

\begin{prop} \label{Michael} Each irreducible unitary representation of $\mathcal{U}(n)$ restricts to an irreducible unitary representation of $\mathrm{SU}(n)$, and all irreducible unitary representations of $\mathrm{SU}(n)$ are obtained in this fashion. Furthermore, two irreducible unitary representations $\pi_{1}$ and $\pi_{2}$ of $\mathcal{U}(n)$ restrict to the same representation of $\mathrm{SU}(n)$ if and only if, for some $j \in \mathbb{Z}$,
$$
\pi_{2}(g)=(\operatorname{det} g)^{j} \pi_{1}(g), \quad \forall g \in \mathcal {U}(n) .
$$
Hence the set of equivalence classes of irreducible unitary representations of $\mathrm{SU}(n)$ is parametrized by
$$
\left\{\left(d_{1}, \ldots, d_{n-1}, 0\right) \in \mathbb{Z}^n: d_{1} \geq d_{2} \geq \cdots \geq d_{n-1} \geq 0\right\}
$$
\end{prop}

Also, recall the Weyl dimension formula for an irreducible unitary representation $\pi$ of $\mathcal U(n)$ with weights: $w_1 \geq \cdots \geq  w_n$, $w_i \in \mathbb Z$, \cite[Theorem 11.4]{Ambar} (see also \cite[Proposition 2.5]{MisraBagchi}),
$$\dim \pi = \prod_{1 \leq j<k \leq n} \frac{w_{j}-w_{k}+k-j}{k-j}.$$

Combining Proposition \ref{Michael} with the Weyl dimension formula, we find all the $d$-dimensional representations of $SU(d)$. The representations of $\mathcal U(d)$ can be then made up from the ones for $SU(d)$ using the relationship between these representations prescribed in Proposition \ref{Michael} as follows. 
The $d$-dimensional (inequivalent, irreducible and unitary) representations of the group $\mathcal U(d)$ are determined by weights of the form: $(\ell+1,\ell, \ldots,\ell)$ and $(m,\ldots,m, m-1)$, $\ell,m \in \mathbb Z$. As noted in \cite[Proposition 22.2]{MTaylor}, the representation $\rho_\ell$ corresponding to 
the weight $(\ell+1,\ell, \ldots,\ell)$ differs from $\rho_0$ by a power of the determinant: $\rho_\ell(u) = (\det(u))^\ell \rho_0(u)$, $u\in \mathcal U(d)$. The representation $\bar{\rho}_m$ corresponding to  $(m,\ldots,m, m-1)$ is similarly related to $\bar{\rho}_0$. We also point out that $\bar{\rho}_0$ is the contragredient of $\rho_0$. 
We claim that  $\rho_\ell$ and $\bar{\rho}_m$ are the only $d$-dimensional irreducible unitary representations of $\mathcal U(d)$ up to unitary equivalence (Lemma \ref{SU(n)class}). We also claim that $SU(d)$ has no irreducible unitary representation of dimension $2, \ldots, d-1$ (Lemma \ref{noSU(n)}). 

It might be that both of these results are well-known, although, we are not able to locate them. However, A. Koranyi in private communication to one of the authors, has provided a very short proof of Lemma \ref{noSU(n)} using Lie algebraic machinery. A little more effort gives a proof of Lemma \ref{SU(n)class} as well, thanks to  A. Khare, E. K. Narayanan, and C. Varughese. 
However, here we give, what we consider to be an elementary proof these assertions.


\begin{lemma}\label{SU(n)class}
Suppose that $c:\mathcal U(d) \to {\rm GL}_d(\mathbb C)$ is an irreducible unitary representation of $\mathcal U(d)$. Then, up to unitary equivalence, either $c(u)=\det(u)^\ell \bar{u}$ or $c(u)= \det(u)^m u$, $\ell, m\in \mathbb Z$.
\end{lemma}

\begin{lemma} \label{noSU(n)}
If $n\in \mathbb N\!:2\leq n \leq d-1$, then there is no $n$-dimensional irreducible unitary representation of $\mathcal U(d)$, or that of $SU(d)$. 
\end{lemma}

B. Bagchi has observed that  Lemma \ref{SU(n)class} and \ref{noSU(n)} can be combined into the following assertion. 

Let $w_1 \geq \cdots  \geq w_d =0$ be integers. Then, either $w_1 = \cdots = w_d = 0$, or $\prod_{\stackrel{1 \leq j<k \leq d}{w_{d}=0}}\big ( 1+ \frac{w_j - w_k}{k-j} \big )  \geq d$.
Equality holds in this inequality if and only if either $w_1 = \cdots =w_{d-1}=1, w_d=0$ or $w_1 =1$ and $w_2 = \cdots = w_d =0$.
The proof is then by induction on the dimension $d$ similar to the proofs we give below.

The first half of Theorem \ref{class} below describing all the quasi-invariant kernels, which transform as in Definition \ref{qinv} via an irreducible $d$-dimensional unitary representation $c$ of $\mathcal U(d)$, is an immediate consequence of  Lemma \ref{SU(n)class} combined with Theorem \ref{thm positive definiteness} (resp. Theorem \ref{qK}) and Theorem \ref{nonnegativity wrt multiplier u}. The second half follows from Lemma \ref{noSU(n)}. 
We would have liked to prove a similar classification theorem for all the $\mathcal U(d)$-homogeneous operators in the class $\mathcal A_d\mathcal U(\mathbb B_d)$. However, unfortunately, such a classification doesn't follow immediately from the theorem below and requires further investigation. 

\begin{theorem} \label{class}
 Let $K:\mathbb B_d \times \mathbb B_d \to \mathcal M_n(\mathbb C)$ be a non-negative definite kernel. 
 \begin{enumerate}
\item[\rm (a)]Suppose that  $n=d$, and $K$ is quasi-invariant under $\mathcal U(d)$ with respect to the multiplier $c$, where $c:\mathcal U(d) \to {\rm GL}_d(\mathbb C)$ is an irreducible unitary representation. Then 
there exists $U\in \mathcal U(d)$ such that $UK(\boldsymbol z, \boldsymbol w) U^*$ is either of the form 
\begin{align*} \sum_{\ell=1}^\infty\big(a_{\ell,1} - a_{\ell,2}\big)\inp{\boldsymbol z}{\boldsymbol w}^{\ell-1}\overline{\boldsymbol w}  \boldsymbol z^\dagger +\sum_{\ell=0}^\infty a_{\ell,2}\inp{\boldsymbol z}{\boldsymbol w}^{\ell} I_d, \; \; \boldsymbol z, \boldsymbol w\in \mathbb B_d, \end{align*}
$\mbox{where}~~a_{\ell,1} \geq 0 ~~\mbox{and}~~~a_{\ell,1}\leq (\ell +1) a_{\ell,2} \mbox{\rm ~for all~} \ell \in \mathbb Z_+$, 
or of the form 
\begin{equation*}
 \sum_{\ell=1}^{\infty} (\tilde{a}_{\ell,1} -\tilde{a}_{\ell,2} ) \inp{\boldsymbol z}{\boldsymbol w}^{\ell-1}\boldsymbol z \overline{\boldsymbol w}^\dagger + \sum_{\ell=0}^{\infty} \tilde{a}_{\ell,2}  \inp{\boldsymbol z}{\boldsymbol w}^\ell I_d, \;\; \boldsymbol z, \boldsymbol w\in \mathbb B_d,\end{equation*}
$ \mbox{where}~~\tilde{a}_{\ell,2} \geq 0$ and $ (d-1)\tilde{a}_{\ell,2} \leq (\ell+d-1)\tilde{a}_{\ell,1}$ for all $\ell \in \mathbb Z_+.$

\item[\rm (b)] If $1 < n < d$, then there is no $n$-dimensional irreducible unitary representation $c$ such that $K$ is  quasi-invariant under $\mathcal U(d)$ with multiplier $c:\mathcal U(d) \to \text{GL}_n(\mathbb C)$. 
\end{enumerate}
\end{theorem}



\subsection{Elementary proof of Lemma \ref{SU(n)class} and of Lemma \ref{noSU(n)}}
\begin{myproof}[Proof of Lemma \ref{SU(n)class}]

We begin the proof with the \emph{claim} that any irreducible unitary representation, up to unitary equivalence, of $SU(d)$ acting on $\mathbb C^d$ are the ones determined by the weights: $(1,0,\ldots,0)$ and $(1,\ldots ,1, 0)$. In other words, we have to show that the only (admissible)  weights $w=(w_1, \ldots, w_{d-1},0)$ for which 
\begin{equation} \label{dimd}
\displaystyle{\prod_{\stackrel{1 \leq j<k \leq d}{w_{d}=0}} }\frac{w_{j}-w_{k}+k-j}{k-j} = d\end{equation}
are of the form: 
$(1,0,\ldots ,0)$ or $(1,1,\ldots, 1,0)$. 

For $d=2$, the claim is evident from the dimension formula. Assume that the claim is valid for  $d-1$, that is, if
$$\displaystyle{\prod_{\stackrel{1 \leq j<k \leq d-1}{w_{d-1}=0}} }\frac{w_{j}-w_{k}+k-j}{k-j} = d-1,$$ then there are only two alternatives for  $w$, namely,  either $w = (1,0,\ldots, 0)$, or $w= (1,\ldots ,1, 0)$.

Let $w=(w_1, \ldots , w_{d-1},0)$ be a weight satisfying the equality in the dimension formula  \eqref{dimd}. Splitting the product in \eqref{dimd}, we have 
\begin{equation}\label{split}
{{\prod_{\stackrel{1 \leq j<k \leq d}{w_d=0}} }}\frac{w_{j}-w_{k}+k-j}{k-j} = {\prod_{{1 \leq j<k \leq d-1}} }\frac{w_{j}-w_{k}+k-j}{k-j} {\prod_{1 \leq j \leq d-1}}\frac{w_{j}+d-j}{d-j}. 
\end{equation}
We shall consider three possibilities, namely,   
\begin{equation} \label{dimdu}
{\prod_{{1 \leq j<k \leq d-1}} }\frac{w_{j}-w_{k}+k-j}{k-j} = d-1\end{equation}
and the two other possibilities of being strictly greater than $d-1$ and less than $d-1$. 
First, consider the case of equality. In this case, the weight $\hat{w}=(w_1, \ldots , w_{d-1})$ satisfying \eqref{dimdu} determines a irreducible unitary representation of $\mathcal U(d-1)$ of dimension $d-1$. 
But this is also the dimension of the irreducible unitary representation of $SU(d-1)$ determined by $(w_1 - w_{d-1}, w_2 - w_{d-1}, \ldots, w_{d-2} -w_{d-1},  0)$. Then by the induction hypothesis, we either have 
$w_1 = w_{d-1} + 1, w_2 = \cdots =w_{d-2} = w_{d-1}$ or 
$w_1=w_2= \cdots =w_{d-2} = w_{d-1}+1$.
Therefore, the weight $w$ of size $d$ must be of the form $(m, m-1, \ldots , m-1,0)$,  or $(m, \ldots ,m, m-1, 0)$, $m\geq 1$. In case of the first alternative, to ensure validity of \eqref{dimd}, we must also have 
$$\frac{d}{d-1} = \displaystyle{\prod_{1 \leq j \leq d-1} }\frac{w_{j}+d-j}{d-j} \Big ( = \frac{(m+d-1) (m+d-3) \cdots (m+2)  \cdot (m+1)\cdot m}{(d-1)(d-2)\cdots 2 \cdot 1 }\Big ).$$
This is possible only if $m=1$ providing one of the two choices in the induction step.  In case of the second alternative,  $w=( m, \ldots,m,m-1,0)$, and we have 
$$\displaystyle{\prod_{1 \leq j \leq d-1} }\frac{w_{j}+d-j}{d-j} = \frac{(m+d-1) (m+d-2) \cdots (m+2)  \cdot m}{(d-1)(d-2)\cdots 2 \cdot 1 }.$$
Since $m\geq 1$, it follows that the smallest possible value of this product is $\frac{d}{2}$ and it is achieved at $m=1$. Thus  it cannot equal $\frac{d}{d-1}$ unless $d=3$. But if $d=3$, and $m=1$, the weight of size $2$ from the induction hypothesis is of the form $(1,0)$.  So, we get nothing new when $d=3$.

Now, if possible, suppose that 
${\prod_{{1 \leq j<k \leq d-1}} }\frac{w_{j}-w_{k}+k-j}{k-j} \geq d$.  
Then we must have 
$$\prod_{1 \leq j \leq d-1}\frac{w_{j}+d-j}{d-j} \leq 1,$$
which is evidently false unless $w_j=0$, $1\leq j \leq d-1$. But if we choose $w=(0,\ldots ,0)$, then we can't have equality in Equation \eqref{dimd}, therefore it is not an admissible choice.

Finally, let us suppose that 
$1\leq {\prod_{{1 \leq j<k \leq d-1}} }\frac{w_{j}-w_{k}+k-j}{k-j} = \ell \leq d-2$. First, if $\ell=1$, the only possible choice of the weight $w$ is $w_1=\cdots = w_{d-1}$. We must then ensure that   
$$\prod_{1 \leq j \leq d-1}\frac{w_j+d-j}{d-j} = d,$$ 
which is possible only if $w_1=\cdots = w_{d-1}=1$. This, together with the choice $w_d=0$ that we have made earlier, proves that $w=(1, \ldots , 1,0)$ providing the second choice in the induction step. 
In particular, the dimension of the representation determined by the weight $(1,1,\ldots, 1,0)$ is $d$. Now, we must establish that there is no other choice of $w$ satisfying \eqref{dimd}. This follows from Lemma \ref{noSU(n)} proved below. It is also easy to verify directly: If $d=2$ or $3$, there is nothing more to be done. If $d > 3$, then  fix $\ell: 2 \leq \ell \leq d-2$, and pick $w$ such that ${\prod_{{1 \leq j<k \leq d-1}} }\frac{w_{j}-w_{k}+k-j}{k-j} = d - \ell$. Having picked $w$, we also need    
$$\frac{d}{d-\ell} =\prod_{1 \leq j \leq d-1}\frac{w_{j}+d-j}{d-j},$$ 
that is, 
$$d! = (w_1+d-1) \cdots (w_\ell + d-\ell) (d-\ell) (w_{\ell+1} + d-\ell -1) \cdots (w_{d-1} +1),$$
which is valid only if $w$ is of the form $(1,\ldots,1, w_\ell=1, 0, \ldots ,0)$. For this choice of $w$, we see that 
$${\prod_{{1 \leq j<k \leq d-1}} }\frac{w_{j}-w_{k}+k-j}{k-j} = \binom{d-1}{\ell},$$
which can't be equal to $\ell$ for any $d > 3$. So, there are no more admissible weights in this case. 
This completes the verification of the induction step and therefore the proof of the claim. Now, the result follows directly from Proposition \ref{Michael}.
\end{myproof}
\begin{myproof}[Proof of Lemma \ref{noSU(n)}]The proof is by induction on the dimension $d$. The base case of $d=3$ is easily verified.
Now, we assume by the induction hypothesis, that there are no irreducible unitary representation such that 
$$2 \leq t:={\prod_{{1 \leq j<k \leq d-1}} }\frac{w_{j}-w_{k}+k-j}{k-j} \leq d-2.$$


Thus the only choice for $t$ is either $t=1$, or $t \geq d-1$. 
To complete the induction step, we have to show that there is no weight $w=(w_1, \ldots , w_{d-1},0)$ such that 
$$2 \leq \ell:=\displaystyle{\prod_{\stackrel{1 \leq j<k \leq d}{w_{d}=0}} }\frac{w_{j}-w_{k}+k-j}{k-j} \leq d-1.$$
If $t=1$, then the only possible choice of the weight $w$ is $w_1=\cdots = w_{d-1}$, say $u$. From Equation \eqref{split}, it follows that $$\prod_{1 \leq j \leq d-1}\frac{u+d-j}{d-j} = \ell.$$
However since the product on the left hand side of the equation above is an increasing function of $u$ and its smallest value is $1$, the next possible value is $d$, it follows that the value $\ell: 2\leq \ell \leq d-1$ is not taken. Now, let $t \geq d-1$ for some $w$. Then from Equation \eqref{split}, we see that  $$\frac{\ell}{t} =\prod_{1 \leq j \leq d-1}\frac{w_{j}+d-j}{d-j}$$ 
to ensure the existence of a $\ell$-dimensional representation. Since $\tfrac{\ell}{t} \leq 1$ while the product on the right hand side of the equation above is greater or equal to $1$, it follows that the two sides can be equal only if $w_1=\cdots = w_{d-1}=0$. But then  $t$ must be equal to $1$ contrary to our hypothesis. 
\end{myproof}

 A. Koranyi has pointed out that $SU(d)$  is a simple Lie group with  discrete center and its Lie algebra $su(d)$ is simple. Therefore any non-trivial homomorphism of it can have at most a discrete null space, i.e., has to be a local isomorphism. So the image of a representation is a closed subgroup of  $\mathcal U(n)$, therefore  must have the same dimension (as a Lie group) as $SU(d)$.  If  $d > n$, then this is not possible proving Lemma \ref{noSU(n)}. 
 
 E. K. Narayanan observed that a proof of Lemma \ref{SU(n)class} follows from the description of the Lie algebra homomorphisms from $su(d)$ to $u(d)$, the Lie algebra of $\mathcal U(d)$. A. Khare and C. Varughese independently of each other have provided the following argument proving Lemma \ref{SU(n)class}: Since $su(d)$ is simple and $u(d) = su(d) \oplus \mathbb R$, it follows that any Lie algebra homomorphism must map $su(d)$ to itself isomorphically. Also, the inequivalent representastions of $su(d)$ are characterized by the outer automorphisms. These are in one to one correspondence with automorphisms of the corresponding Dynkin diagram. The Dynkin diagram of $su(d)$ is $A_{(d-1)}$ consisting of $d-1$ dots connected by single lines.  For $d > 2$, the (graph) automorphism group of $A_{(d-1)}$ is of order $2$ (identity and a reflection). It follows that there are at most two inequivalent irreducible unitary representations of $SU(d)$, $d\geq 2$. 

We believe, it will be interesting to find an answer to the two questions:  (a) What  possible values $\dim \pi$ can take if $d$ is fixed. (b) If $d$ and $n=\dim \pi$ are fixed, how many $n$-dimensional inequivalent irreducible unitary representations are there of the group $SU(d)$.

\subsubsection*{Acknowledgment} The authors are grateful to E. K. Narayanan for several lectures explaining the parametrization and realization of the irreducible unitary representations of $\mathcal U(n)$ and for going through some of the proofs carefully. We also thank Sameer Chavan for several comments on a preliminary draft of this paper.




\end{document}